\documentclass[12pt]{article}
\usepackage[dvipsnames]{xcolor}

\usepackage{graphicx,tikz,cancel}
\usepackage[dvipsnames]{xcolor}

\usepackage{amsmath,amssymb,amsthm,amsfonts,amsopn,mathrsfs}
\usepackage{bm}
\usepackage{mathtools}
\usepackage{bigints}
\usepackage{latexsym}
\usepackage{color}
\usepackage{pgfplots}
\pgfplotsset{width=8cm,compat=1.15}

\usepackage{caption,subcaption,listings,multirow,booktabs}
\usepackage[algoruled,lined,linesnumberedhidden,noend]{algorithm2e}
\DontPrintSemicolon

\usepackage{dsfont,microtype,mathrsfs,bigints,stmaryrd,hyperref,url}

\usepackage[english]{babel}

\usepackage[square,numbers,comma,sort&compress]{natbib}

\makeatletter
\newcommand{\citethm}[2][]{%
\begingroup
  \let\NAT@mbox=\mbox
  \let\@cite\NAT@citenum
  \let\NAT@space\NAT@spacechar
  \let\NAT@super@kern\relax
  \renewcommand\NAT@open{[}%
  \renewcommand\NAT@close{]}%
  \cite[#1]{#2}%
  \endgroup
}

\newcommand{\RR}{\mathbb{R}}

\newcommand{\EE}{\mathbb{E}}

\newcommand{\de}{\mathrm{d}}

\newcommand{\abs}[1]{\left\lvert #1 \right\rvert}

\newcommand{\norm}[1]{\left\lVert#1\right\rVert}

\newcommand{\argmax}{\text{argmax}}

\newcommand{\mc}[1]{\ensuremath{\mathcal{{#1}}}}

\newcommand{\ie}{\textit{i.e. }}
\newcommand{\eg}{\textit{e.g. }}

\DeclareSymbolFont{calletters}{OMS}{cmsy}{m}{n}
\DeclareSymbolFontAlphabet{\mathcal}{calletters}

\def\be{\begin{eqnarray}}
\def\ee{\end{eqnarray}}

\def\b*{\begin{eqnarray*}}
\def\e*{\end{eqnarray*}}

\newtheorem{Theorem}{Theorem}[section]

\newtheorem{Proposition}[Theorem]{Proposition}

\newtheorem{Assumption}{Assumption}

\newtheorem{Remark}[Theorem]{Remark}
\newtheorem{Example}[Theorem]{Example}

\makeatletter \@addtoreset{equation}{section}

\usepackage{dsfont}

\newcommand{\bea}{\begin{eqnarray}}
\newcommand{\bes}{\begin{subequations}}
\newcommand{\ees}{\end{subequations}}
\newcommand{\bgt}{\begin{gather}}
\newcommand{\egt}{\begin{gather}}
\newcommand{\eea}{\end{eqnarray}}
\newcommand{\beaa}{\begin{eqnarray*}}
\newcommand{\eeaa}{\end{eqnarray*}}

\def \D{\mathbb{D}}
\def \E{\mathbb{E}}
\def \F{\mathbb{F}}

\def \P{\mathbb{P}}

\def \R{\mathbb{R}}

\def\Ac{{\cal A}}

\def\Fc{{\cal F}}

\def \eps{\varepsilon}

\def \0{\mathbf{0}}

\def \vp{\varphi}
\def \x{\times}

\def\1{{\rm 1}}

  \def\vs#1{\vspace{2mm}}

\usepackage[left=0.75in, right=0.75in, top=1in, bottom=1.5in]{geometry}
\usepackage[T1]{fontenc}
\usepackage{hyperref}

\def\D{{\mathbb D}}
\def\R{{\mathbb R}}
\def\P{{\mathbb P}}

\def\Ab{{\mathbb A}}
\def\Ac{{\mathcal A}}

\def\ar{{\rm a}}

\usepackage{enumitem}
\title{Diffusive limit approximation of pure-jump  optimal stochastic control problems}
\author{Marc Abeille\footnote{Criteo AI Lab. m.abeille@criteo.com}, Bruno Bouchard\footnote{CEREMADE, Universit\'e Paris-Dauphine, PSL, CNRS.  bouchard@ceremade.dauphine.fr. }, Lorenzo Croissant\footnote{CEREMADE, Universit\'e Paris-Dauphine, PSL, CNRS, and Criteo AI Lab.  croissant@ceremade.dauphine.fr.} }

\newcommand{\lc}[1]{#1}
\def\bru#1{{#1}}
\def\bblue#1{{#1}}

\begin{document}

\maketitle

\begin{abstract} We consider the diffusive limit of a typical pure-jump Markovian control problem as the intensity of the driving Poisson process tends to infinity. We show that the convergence speed is provided by the H\"older exponent of the Hessian of the limit problem, and explain how correction terms can be constructed. This provides an alternative efficient method for the numerical approximation of the optimal control of a pure-jump problem in situations with very high intensity of jumps. We illustrate this approach in the context of a display advertising auction problem.  
\end{abstract}

  
\section{Introduction}

Let $N$ be a random point process with predictable compensator $\lambda \nu(\de e)\de t$, for some \bru{probability} measure $\nu$ on $\R$, $\lambda>0$, and let $X^{t,x,\alpha}$ be the solution of 
\[
X^{t,x,\alpha}=x+  \int_t^\cdot \int b(X^{t,x,\alpha}_{s-},\alpha_{s},e)N(\de e,\de s)\,, 
\]
 in which $\alpha$ belongs to  the set $\Ac$ of predictable controls with values in some given set $\Ab$. Then, under mild assumptions, the value of the control problem 
 \[
V(t,x):=\sup_{\alpha \in \Ac}  \E\left[ \int_{t}^{ T}   r(X^{t,x,\alpha}_{s-},\alpha_{s})\de N_{s}\right]\,,
\]
with $N_t:=N(\RR,[0,t])$, $t\ge 0$, solves the integro-differential equation 
\begin{align}\label{eq: pde integro intro} 
\partial_{t}V + \lambda \sup_{a\in \Ab}\left( \int V(\cdot,\cdot+b(\cdot,a,e))\nu(\de e)-V+r(\cdot,a)\right)=0 \mbox{ on }[0,T)\times\RR
\end{align}
with boundary condition $V(T,\cdot)=0$, possibly in the sense of viscosity solutions. From this characterization, standard numerical schemes follow that allow one to approximate both the value function $V$ and the associated optimal control. 

However, \eqref{eq: pde integro intro}  is non-local and obtaining a precise approximation of the solution is highly time consuming as soon as the intensity $\lambda$ of $N$ is large.   
This is the case, for instance, for ad-auctions on the web, {see e.g.~\cite{fernandez-tapia2017Optimal}}, that are posted almost in continuous time, and on which one would typically like to apply reinforcement learning technics based on the resolution of \eqref{eq: pde integro intro} for the current estimation of the parameters, leading to a possibly large number of resolutions  for different sets of parameters. 
On the other hand, when $\lambda$ is very large, it is tempting to approximate the original jump diffusion control problem by its asymptotic as $\lambda\to \infty$. In this paper, we consider the diffusive limit approximation. Namely, if one takes $\lambda$ of the form $\lambda=1/\epsilon$, with $\epsilon$ small, and $b=\epsilon b_{1}+\sqrt{\epsilon} b_{2}$ with 
$\int b_{2}(\cdot,e)\nu(\de e)=0$, then a second order Taylor expansion on \eqref{eq: pde integro intro}  implies that $\epsilon V$ converges as $\epsilon \to 0$ to the solution $\bar V$ of 
\begin{align}\label{eq: HJB diffusion intro} 
\partial_{t} \bar V +\sup_{\bar a\in \Ab}\left(\int b_{1}(\cdot,\bar a,e)\nu(\de e) \partial_{x} \bar V  +\frac12 \int \abs{b_{2}}^{2}(\cdot,\bar a,e)\nu(\de e)  \partial^{2}_{xx}\bar V +r(\cdot,\bar a)\right)=0, 
\; \bar V(T,\cdot)=0.
\end{align}
The advantage of the above is that it is now a local equation which can be solved in a much more efficient way. Note that another possibility is to consider a first order expansion as in \cite{fernandez-tapia2017Optimal}, which corresponds to considering a fluid limit, but this is less precise.

For such a specification of the coefficients ($\lambda$, $b$), the existence of a diffusive limit is expected, see e.g.~\cite{jacod2013limit} for general results on the convergence of stochastic processes. For control problems, the convergence of the value function can be proved by using the stability of viscosity solutions as in  \cite[Section 3]{fleming1989existence}, which considers the limit of discrete time zero-sum games, or by applying weak-convergence results. In particular an important literature on this subject exists within the insurance and queueing network literatures, see \eg~\cite{bauerle2004approximation,cohen2020rate,chen2001fundamentals}. However, it seems that there is no general result on the speed of convergence in the case of a (generic) optimal control problem as defined in Section \ref{sec : pure jump problem} below. 

In Section \ref{sec: diffusive limit}, we  verify that the above intuition is correct. Unlike \cite{fleming1989existence}, we do not simply rely on the stability of viscosity solutions. Neither do we rely on the weak convergence of the underlying process. The reason is that weak convergence does not give access to the convergence speed in optimal control problems. Instead, we directly study the regularity of the solution to \eqref{eq: HJB diffusion intro}. Thanks to its vanishing terminal condition (otherwise it should be assumed smooth enough), we show that $\partial^{2}_{xx}\bar V$ is uniformly $\beta$-H\"older in space, for some $\beta\in (0,1]$, whenever the coefficients of \eqref{eq: HJB diffusion intro} are uniformly Lipschitz in space and under a uniform ellipticity condition. By a second order Taylor expansion, this allows us to pass from \eqref{eq: HJB diffusion intro}  to \eqref{eq: pde integro intro} up to an error term of order $\epsilon^{\frac{\beta}2}$, and therefore provides the required convergence rate. In general this rate \lc{cannot} be improved. As a by-product, we obtain an easy way to construct an $\epsilon^{\frac{\beta}{2}}$-optimal control for the original pure-jump control problem. We then study the limit $\epsilon^{-\frac{\beta}{2}}(V-\bar V)$ as $\epsilon\to 0$. Under mild assumptions, we show that it solves a (possibly non-linear) PDE. This provides a first error correction term. To achieve higher orders of convergence, this approach can be generalised to a system of non-linear PDEs, upon its existence. 

As an example of application, we consider in Section \ref{sec: exemple auction} a simplified repeated online auction bidding problem, where a buyer seeks to maximise his profit when facing both competition and a seller who adapts his price to incoming bids. Our numerical experiments show that our approximation permits a considerable gain in computation time.

For ease of exposition, we shall restrict to situations where the controlled process is of dimension one. This fact will be used explicitly only to derive our regularity results in Section \ref{subsec: regul bar V}. Similar results can be obtained in higher dimension, by using  standard regularity results for parabolic partial differential equations, see e.g.~\cite{lieberman1996second,ladyzhenska1988linear}.

\section{The pure-jump optimal control problem}\label{sec : pure jump problem}

In this section, we  begin by providing the definition of our pure-jump control problem, and state the well-known link with its associated Hamilton-Jacobi-Bellman equation. The properties stated below are elementary but will be useful for the derivation of our main approximation result of Section \ref{sec: diffusive limit}.

\subsection{Definition}\label{subsec: prelims jump problem}

  Let $\Omega=\D$ denote the space of one dimensional c\`adl\`ag functions on $\R_{+}$ and ${\cal M}(\R\times \R_+)$ denote the collection of positive finite measures on $\R\times \R_+$.  Consider a  measure-valued map  $N: \D\mapsto {\cal M}(\R\times \R_{+})$   and  a probability measure  $\P$ on  $\D$ such that   $N$ is a continuous real-valued $\R$-marked point process with compensator $\lambda\nu(\de e)\de t$, in which   $\lambda>0$ and $\nu$ is a probability measure on $\R$. See e.g.~\cite{bremaud1981point}. For ease of notations, we set $N_{t}:=N(\R,[0,t])$ for $t\ge 0$.

  Let $\F^{t}=(\Fc^{t}_{s})_{s\ge t}$ be the $\P$-augmentation of the filtration generated by $(\bru{\int_{t}^{s} e N(de,dr)})_{s\ge t}$. Given a compact subset $\Ab$ of $\R$, we let $\Ac^{t}$ be the collection of $\F^{t}$-predictable processes with values in $\Ab$. For ease of notations, we also define $\Ac:=\cup_{t\ge 0}\Ac^{t}$.
Throughout this paper, unless otherwise stated we will work on the filtered probability space $(\Omega,\Fc, \F, \P)$, where $\Fc=\Fc^0_{T}$ for $T>0$ given and $\F=\F^0$. 

 We now consider a bounded measurable map $(x,a,e) \in \R\x \Ab\times \R \mapsto b(x,a,e)$. Given $(t,x)\in \R_+\x \R$ and $\alpha \in \Ac$, we  define the c\`adl\`ag process $X^{t,x,\alpha}$ as the solution of 
\begin{align}\label{def : X}
X^{t,x,\alpha}=x+  \int_t^\cdot \int b(X^{t,x,\alpha}_{s-},\alpha_{s},e)N(\de e,\de s).
\end{align}

\par Given  a bounded measurable map $(x,a) \in \R\x \Ab\mapsto r(x,a)\in \R$, we consider  the expected gain function  
\begin{align}\label{def : V}
(t,x,\alpha)\in  [0,T]\times \R\times \Ac \mapsto  J(t,x;\alpha) &:=\EE\left[ \int_{t}^{ T}   r(X^{t,x,\alpha}_{s-},\alpha_{s})\de N_{s}\right],
\end{align}
together with the value function 
\begin{align}\label{def : hat V}
V(t,x):=\sup_{\alpha \in \Ac^{t}} J(t,x;\alpha),\;(t,x)\in [0,T]\x \R.
\end{align}
All throughout the paper, we make the following standard assumption, which will in particular ensure that $V$ is the unique (bounded) viscosity solution of the associated Hamilton-Jacobi-Bellman equation, see Proposition \ref{prop: HJB saut + continuite} below.

\begin{Assumption}\label{asmp: cont coefs}
For each $e\in \R$, $(x,a)\in \R\x \Ab\mapsto (b(x,a,e), r(x,a))$ is  continuous. Moreover, $(b,r)$ is bounded.
\end{Assumption}

\begin{Remark} Note that boundedness of the coefficients $b$ and $r$ is not essential in the following arguments. One could assume only linear growth in space, uniformly in the control.  We make the above (strong) assumptions to avoid unnecessary complexities. 
\end{Remark}

\subsection{Dynamic programming equation and optimal Markovian control}

Let us now state the well-known characterization of $V$ in terms of the theory of viscosity solutions. 

As usual, we say that a lower-semicontinuous (resp.~upper-semicontinuous) locally bounded map  $U: [0,T]\x \R\mapsto \R$ is a viscosity supersolution (resp.~subsolution) of
\begin{align}
&\partial_{t}\vp + \sup_{a\in \Ab}\left( \int \vp(\cdot,\cdot+b(\cdot,a,e))\nu(\de e)-\vp+r(\cdot,a)\right)\lambda =0,\; \mbox{ on } [0,T)\times \R,\label{eq: PDE HJB Jump}
\end{align}
if for all $(t,x)\in [0,T)\x \R$ and all $C^{1}$ functions $\vp :[0,T]\x \R\mapsto \R$ such that $(t,x)$ attains a minimum (resp.~maximum) of $U-\vp$ on $ [0,T)\x \R$ we have 
\begin{align*}
&\kappa\left\{\partial_{t}\vp(t,x) + \sup_{a\in \Ab}\left( \int U(t,x+b(x,a,e))\nu(\de e)-U(t,x)+r(x,a)\right)\lambda\right\}\le 0
\end{align*}
with $\kappa=1$ (resp.~$\kappa=-1$). 

\begin{Proposition}\label{prop: HJB saut + continuite} $V$ is a continuous and bounded viscosity solution of \eqref{eq: PDE HJB Jump} such that 
\begin{align}\label{eq: cond bord HJB saut} 
\lim_{t'\uparrow T,x'\to x} V(t',x')=0,\;x\in \R.
\end{align}
Moreover, comparison holds for \eqref{eq: PDE HJB Jump} in the class of bounded functions.
\end{Proposition}

\begin{proof} The argument being standard, we only sketch it. \bru{First note that the continuity at $T$ follows immediately from the fact that $r$ is bounded, namely $|V(t,\cdot)|\le (T-t)\norm{r}_\infty$ for $t\le T$.}
Fix $h\in (0,T-t]$, \bru{$t\le T$ and $x\in \R$}. Let $\tau^{t}_{1}$ be the first jump of $N$ after time $t$. Denote by $V_*$ and $V^*$ the lower- and upper-semicontinuous envelopes of $V$, \ie
\[V_{*}(t',x'):=\liminf_{(s,y)\to (t',x')} V(s,y)\;,\;V^{*}(t',x'):=\limsup_{(s,y)\to (t',x')} V(s,y)\,.\]
It follows from the same arguments as in \cite{bouchard2011weak} that $V$ satisfies the (weak) dynamic programming principle 
\begin{align}
 &\sup_{\alpha \in \Ac^{t}} \E\left[V_*(\tau_{1}^{t}\wedge h,X^{t,x,\alpha}_{\tau^{t}_{1}\wedge h})+r(X^{t,x,\alpha}_{\tau^{t}_{1}-},\alpha_{\tau^{t}_{1}})\1_{\{\tau^{t}_{1}\le h\}} \right] \nonumber \\
 &\le V(t,x) \label{eq: DPP Vtheta}\\
 &\le \sup_{\alpha \in \Ac^{t}} \E\left[V^{*}(\tau_{1}^{t}\wedge h,X^{t,x,\alpha}_{\tau^{t}_{1}\wedge h})+r(X^{t,x,\alpha}_{\tau^{t}_{1}-},\alpha_{\tau^{t}_{1}})\1_{\{\tau^{t}_{1}\le h\}} \right]\,. \nonumber
\end{align}
Following \cite{bouchard2011weak} again and using \cite[Lemma 22]{bouchard2002stochastic}, this implies that $V_{*}$ and $V^{*}$ are, respectively, a super- and a subsolution in the viscosity sense of \eqref{eq: PDE HJB Jump}. Since $r$ and $b$ are bounded, the map $(t,x)\mapsto (1+x^{2})e^{-Ct}$ is also a viscosity supersolution of the above, as soon as $C>0$ is large enough. Standard arguments then imply that comparison holds for the above Hamilton-Jacobi-Bellman equation in the class of bounded  functions (or even with linear growth), and therefore that $V_{*}=V^{*}$, meaning that $V$ is continuous. 
\end{proof}

We next prove the existence of an optimal Markovian control. In the following, we denote by $\mathfrak{A}$ the collection \lc{of} $\Ab$-valued Borel maps on $\R_+\x \R$.

\begin{Proposition}\label{prop: existence control optimal Poisson} For all $(t,x)\in \R_+\times \R$, there exists $\hat \alpha[t,x] \in \Ac^{t}$ such that  
$V(t,x)=J(t,x;\hat \alpha[t,x] )$.
It takes the form  
\[\hat \alpha[t,x]=\sum_{i\ge 0} \1_{(\tau^{t}_{i},\tau^{t}_{i+1}]}\hat{\ar}(\cdot,X^{t,x,\hat \alpha[t,x]}_{\tau^{t}_{i}}) \]
in which $\tau^{t}_{i}$ is the $i$-th jump of $N$ after time $t$\bru{,} for $i\ge 1$, with $\tau^{t}_{0}:= t$, and $\hat \ar \in \mathfrak{A}$ satisfies 
\[\hat{\ar}(t',x') \in \underset{a\in \Ab}{\rm argmax}\left\{ \int V(t',x'+b(x',a,e))\nu(\de e)+r(x',a) \;\right\},\;(t',x')\in \R_{+}\x \R\,.\]
\end{Proposition}

\begin{proof} Since  $V, b$ and $r$  are continuous, by Proposition \ref{prop: HJB saut + continuite} and Assumption \ref{asmp: cont coefs}, and since $\Ab$ is compact, we can find a Borel measurable map $(t,x)\mapsto \hat{\ar}(t,x)$ such that 
$
\hat{\ar}(t,x)$ belongs to $\argmax \{\int V(t,x+b(x,a,e))\nu(de)$ $+$ $r(x,a),$ $\;a\in \Ab \}$ for all $(t,x)$, 
see e.g.~\cite[Proposition 7.33, p.153]{BertsekasShreve.78}. Let us fix $(t_{0},x_{0})\in [0,T]\x \R$. By the dynamic programming principle in \eqref{eq: DPP Vtheta}, the continuity of $V$, and the definition of $\hat \ar$ above, 
\begin{align*}
V(t_{0},x_{0})
&=\sup_{\alpha \in \Ac^{t_{0}}} \E\left[V(\tau_{1}^{t_{0}}\wedge T,X^{t_{0},x_{0},\alpha}_{\tau^{t_{0}}_{1}\wedge T})+r(X^{t_{0},x_{0},\alpha}_{\tau^{t_{0}}_{1}-},\alpha_{\tau^{t_{0}}_{1}}) \1_{\{\tau^{t_{0}}_{1}\le T\}}\right] \\
&=\sup_{\alpha \in \Ac^{t_{0}}} \E\left[\left(\int V(\tau_{1}^{t_{0}}\wedge T, x_{0}+b(x_{0},\alpha_{\tau^{t_{0}}_{1}\wedge T},e))\nu(\de e)+r(x_{0},\alpha_{\tau^{t_{0}}_{1}})\right)\1_{\{\tau^{t_{0}}_{1}\le T\}} \right] \\
&=\E\left[\left(\int V(\tau_{1}^{t_{0}}\wedge T,x_{0}+b(x_{0},\hat \alpha_{\tau^{t_{0}}_{1}\wedge T},e)) \nu(\de e)+r( x_{0},\hat \alpha_{\tau^{t_{0}}_{1}}) \right)\1_{\{\tau^{t_{0}}_{1}\le T\}}\right] \\
&=\E\left[V(\tau_{1}^{t_{0}}\wedge T,X^{t_{0},x_{0},\hat \alpha}_{\tau^{t_{0}}_{1}\wedge T})+r(X^{t_{0},x_{0},\hat \alpha}_{\tau^{t_{0}}_{1}-},\hat \alpha_{\tau^{t_{0}}_{1}})\1_{\{\tau^{t_{0}}_{1}\le T\}} \right] 
\end{align*}
in which $\hat{\alpha}:=\hat{\ar}(\cdot,x_{0})\1_{(t_{0},\tau^{t_{0}}_{1}]}$. 

\noindent For ease of notations, we now set $\vartheta_{1}:=\tau_{1}^{t_{0}}\wedge T$ and $X_{1}:=X^{t_{0},x_{0},\hat \alpha}_{\tau^{t_{0}}_{1}\wedge T}$. 
By the same reasoning as above, we have, for a fixed $\omega\in \Omega$, 
\begin{align*}
V(\vartheta_{1}(\omega),X_{1}(\omega))&=\E\left[V(\tau_{1}^{\vartheta_{1}(\omega)}\wedge T,X^{\vartheta_{1}(\omega),X_{1}(\omega),\hat \alpha(\omega)}_{\tau^{\vartheta_{1}(\omega)}_{1}\wedge T})+r(X^{\vartheta_{1}(\omega),X_{1}(\omega),\hat \alpha(\omega)}_{\tau^{\vartheta_{1}(\omega)}_{1}-},\hat \alpha_{\tau^{\vartheta_{1}(\omega)}_{1}}(\omega))\1_{\{\tau^{\vartheta_{1}(\omega)}_{1}\le  T\}} \right]
\end{align*}
in which 
\[\hat{\alpha}(\omega): =\hat{\ar}(\cdot,x_{0})\1_{(t_{0},\tau^{t_{0}}_{1}(\omega)]}+\hat{\ar}(\cdot,X_{1}(\omega))\1_{(\tau^{t_{0}}_{1}(\omega),\tau^{\vartheta_{1}(\omega)}_{1}]}.\]
The right-hand side of the above coincides $\P$-a.e. with 
\begin{align*}
&\E\left[V(\tau_{1}^{\vartheta_{1}}\wedge T,X^{\vartheta_{1},X_{1},\hat \alpha}_{\tau^{\vartheta_{1}}_{1}\wedge T})+r(X^{\vartheta_{1},X_{1},\hat \alpha}_{\tau^{\vartheta_{1}}_{1}-},\hat \alpha_{\tau^{\vartheta_{1}}_{1}})\1_{\{\tau^{\vartheta_{1}}_{1}\le T\}}\bigg\vert\Fc_{\vartheta_{1}} \right]
\\
&=\E\left[V(\tau_{2}^{t_{0}}\wedge T,X^{t_{0},x_{0},\hat \alpha}_{\tau^{t_{0}}_{2}\wedge T})+r(X^{t_{0},x_{0},\hat \alpha}_{\tau^{t_{0}}_{2}-},\hat \alpha_{\tau^{t_{0}}_{2}})\1_{\{\tau^{t_{0}}_{2}\le T\}}\bigg\vert\Fc_{\tau^{t_{0}}_{1}\wedge T} \right].
\end{align*}

\noindent Let us  complete the definition of $\hat \alpha$ by now letting it be defined by 
\[\hat \alpha=\sum_{i\ge 0} \1_{(\tau^{t_{0}}_{i},\tau^{t_{0}}_{i+1}]}\hat{\ar}(\cdot,X^{t_{0},x_{0},\hat \alpha}_{\tau^{t_{0}}_{i}}).\]
By iterating the above procedure, we have 
\begin{align*}
V(t_{0},x_{0})
&=\E\left[V(\tau_{n}^{t_{0}}\wedge T,X^{t_{0},x_{0},\hat \alpha}_{\tau^{t_{0}}_{n}\wedge T})+\int_{t_{0}}^{\tau^{t_{0}}_{n}\wedge T}r(X^{t_{0},x_{0},\hat \alpha}_{s-},\hat \alpha_{s})\de N_{s}\right],\;n\ge 1.
\end{align*}
Since $\tau^{t_{0}}_{n}\to \infty$ $\P$-a.s. as $n\to \infty$, it now follows from the dominated convergence theorem \bru{and \eqref{eq: cond bord HJB saut}} that 
\begin{align*}
V(t_{0},x_{0})
&=\E\left[\int_{t_{0}}^{T}r(X^{t_{0},x_{0},\hat \alpha}_{s-},\hat \alpha_{s})\de N_{s}\right]. 
\end{align*}
\end{proof}

\section{Diffusive approximation}\label{sec: diffusive limit}
 
As already mentioned, the characterization of Propositions \ref{prop: HJB saut + continuite} and \ref{prop: existence control optimal Poisson} allows one to estimate numerically the value function and the associated optimal control. However, the integro-differential equation \eqref{eq: PDE HJB Jump} is non-local and the computational cost of its numerical resolution increases as $\lambda$ grows. On the other hand, we can expect that our pure-jump problem admits a diffusive limit as $\lambda\to \infty$ which is, by its local nature, much easier to solve numerically, and can serve as a good proxy of the original problem as soon as $\lambda$ is large enough. 

In this section, we begin by defining the diffusion control problem that is the candidate for the diffusive limit of our pure-jump problem. We then study the regularity of the corresponding value function, from which we will be able to derive our main approximation result, see Theorem \ref{thm : ecart V et bar V} below, and construct approximate optimal controls, see Proposition \ref{prop: construction control eps opti}.  Finally, we identify a first order correction term in Subsection \ref{subsec: first order correction}, which is extended to higher orders in Subsection \ref{subsec: higher order expansions}.
 
\subsection{The candidate diffusive limit}

Given $\epsilon\in (0,1)$, we now take as $\lambda$  the intensity   
$$ 
\lambda_{\epsilon}:=\epsilon^{-1}
$$
so that it is large for $\epsilon>0$ small. To ensure the existence of a diffusive limit, we need to assume that the jump coefficient $b$ introduced in Section \ref{sec : pure jump problem} is of the form
$$
b_{\epsilon}=\epsilon b_{1}+\sqrt{\epsilon} b_{2}\,
$$
for  two bounded measurable maps $b_1,b_{2}: \R\x \Ab\times \R\mapsto \R$,  each satisfying    Assumption \ref{asmp: cont coefs} (with $b_{i}$ in place of $b$, $i=1,2$), and with $b_2$ satisfying the additional Assumption \ref{asmp: b2 conds}.
\begin{Assumption}\label{asmp: b2 conds}
The function $b_2$ satisfies:
\begin{align}\label{eq: hyp centrage et ellipticite b2}
 \int b_{2}(x,a,e) \nu(\de e)=0 \;\mbox{ for all $(x,a)\in \R\x \Ab$, }\mbox{ and }  \inf_{(x,a)\in \R\x \Ab} \int \abs{b_{2}(x,a,e)}^{2} \nu(\de e)=: \eta>0.
\end{align}
\end{Assumption}
In the above, the coefficient $b_{1}$ should be interpreted as a drift term while $b_{2}$ is a volatility. The respective scaling in $\epsilon$ and $\sqrt{\epsilon}$ together with Assumption \ref{asmp: b2 conds} are required to ensure that our pure-jump problem actually admits a diffusive limit of the form \eqref{eq: def SDE bar X} below. Problems where this scaling of coefficient is appropriate involve many jumps of small relative size, with a variance of the same order as their drift over time.

\noindent Likewise, we consider the value function 
\begin{align}\label{def : hat V eps}
V_{\epsilon}(t,x) :=\sup_{\alpha\in \Ac^{t}} J_{\epsilon}(t,x;\alpha)\mbox{ with }\;J_{\epsilon}(t,x;\alpha):=\frac{1}{\lambda_{\epsilon}}\E\left[ \int_{t}^{ T}   r(X^{t,x,\alpha}_{s-},\alpha_{s})\de N_{s}\right],\; (t,x)\in [0,T]\x \R.
\end{align}
Note that the scaling by $1/\lambda_{\epsilon}$ means that (up to a constant factor $T-t$) we consider the gain by average unit of actions on the system. Indeed $\E[N_{T}-N_{t}]=\lambda_{\epsilon}(T-t)$ and the control applies only at jump times of $N$. Note that we omit the dependence of $N$ on $\epsilon$, for ease of notations. 
\vs2

We shall see that $V_\epsilon$, together with the associated optimal policy, can be approximated by considering its diffusive limit as $\epsilon\to 0$. The coefficients of the associated Brownian diffusion SDE are given by:
\begin{align*}
\mu(x,a):=  \int b_{1}(x,a,e) \nu(\de e),\; \sigma(x,a):= \left( \int \abs{b_{2}(x,a,e)}^{2} \nu(\de e)\right)^{\frac12},\;(x,a)\in \R\x \Ab.
\end{align*}

From now on, we assume that they satisfy the following. 
\begin{Assumption}\label{asmp: Lipschitz diffusion coefs}
The maps $x\in \R \mapsto\mu(x,a)$, $x\in \R \mapsto\sigma(x,a)$ and  $x\in \R \mapsto r(x,a)$  are  Lipschitz, uniformly in $a\in\Ab$, with respective Lipschitz constants $\norm{\mu}_{\rm Lip}$, $\norm{\sigma}_{\rm Lip}$ and $\norm{r}_{\rm Lip}$.
\end{Assumption}%
\noindent More precisely, let $\bar \P$ be a probability measure on $\D$ and let $W$ be a stochastic process such that $W$ is a $\bar \P$-Brownian motion,  let $\bar \F^{t}=(\bar \Fc^{t}_{s})_{s\ge 0}$
be the $\bar \P$-augmentation of the filtration generated by $(W_{\cdot \vee t}-W_{t})$, and let
$\bar \Ac^{t}$ be the collection of $\bar \F^{t}$-predictable processes. Given $\bar \alpha \in \bar \Ac^{t}$, we  can then define $\bar X^{t,x,\bar\alpha}$ as the unique strong  solution of 
\begin{align}\label{eq: def SDE bar X}
\bar X^{t,x,\bar\alpha}=x+\int_{t}^{\cdot} \mu(\bar X^{t,x,\bar\alpha}_{s},\bar\alpha_{s})\de s +\int_{t}^{\cdot} \sigma(\bar X^{t,x,\bar\alpha}_{s},\bar\alpha_{s})\de W_{s}.
\end{align}
The candidate diffusive limit problem is then defined as
\begin{align*}
\bar V(t,x):=\sup_{\bar\alpha \in \bar \Ac^{t}} \bar \E\left[\int_{t}^{T}r(\bar X^{t,x, \bar\alpha}_{s}, \bar\alpha_{s}) \de s\right],\;(t,x)\in [0,T]\x \R
\end{align*}
where $\bar \E$ is the expectation operator under $\bar \P$.

\subsection{Regularity properties}\label{subsec: regul bar V}

We first prove  that $\bar V$ is a smooth solution of its associated Hamilton-Jacobi-Bellman equation. Most importantly, its second order space derivative is $\beta$-H\"older continuous, for some $\beta \in (0,1]$. This will allow us, in Section \ref{subsec : conv speed} below, to prove that it actually coincides with the diffusive limit of $V_{\epsilon}$ as $\epsilon$ vanishes. The precise value of the H\"older exponent $\beta$ will be further discussed in Remark \ref{rem : sur choix beta} below.

\begin{Proposition}\label{prop: HJB bar V} $\bar V$ belongs to $C^{1,2}_{b}([0,T\bru{)}\x \R)\bru{\cap C^{0}([0,T\bru{]}\x \R)}$ and is the unique bounded solution of 
\begin{align}
&\partial_{t} \bar V +\sup_{\bar a\in \Ab}\left(\mu(\cdot,\bar a)\partial_{x} \bar V  +\frac12 \sigma^{2}(\cdot,\bar a) \partial^{2}_{xx}\bar V +r(\cdot,\bar a)\right)=0,\; \mbox{on } [0,T)\x \R,\label{eq: PDE bar V}\\
&\bar V(T,\cdot)=0,\; \mbox{on } \R. \label{eq: cond T bar V}
\end{align}
Moreover, there exists $\beta \in (0,1]$, such that $\partial^{2}_{xx} \bar V$ is (uniformly) $\beta$-H\"older continuous in space \bru{on $[0,T)\x \R$}.
\end{Proposition}

\begin{proof}%

\begin{enumerate}[label=\rm\alph*)]

\item We first show that $\bar V\in C^{1,2}_{b}([0,T)\x \R)\cap C^{0}([0,T]\x \R)$. \bru{Note that the continuity at $T$ follows again form the fact that $r$ is bounded: $|\bar V(t,\cdot)|\le (T-t)\|r\|_{\infty}$ for $t\le T$.} Let us set 
\[F(x,p,q) :=\sup_{\bar a\in \Ab}\left(\mu(x,\bar a) p  +\frac12 \sigma^{2}(x,\bar a) q +r(x,\bar a)\right),\; (x,p,q)\in \R^{3}\,,\]
and observe that, by Assumptions \ref{asmp: b2 conds} and \ref{asmp: Lipschitz diffusion coefs}, 
\begin{align}
\frac12\eta \abs{q-q'}\le& \abs{F(x,p,q)-F(x,p,q')}\le \frac12 \norm{\sigma}^{2}_{\infty}\abs{q-q'} \label{eq: F unif elliptic hyp 14.19 et 14.45}\\
vF(x,0,0)\le& \norm{r}_{\infty}(1+\abs{v}^{2})   \label{eq: hy 14.20}\\
\abs{F(x,p,q)-F(x',p',q')}\le& (\abs{p} \norm{\mu}_{\rm Lip}+\abs{q} \norm{\sigma}_{\infty}\norm{\sigma}_{\rm Lip}+\norm{r}_{\rm Lip})\abs{x-x'} \nonumber\\
&+ \norm{\mu}_{\infty}\abs{p-p'}+ \frac12 \norm{\sigma}_{\infty}^{2}\abs{q-q'} \label{eq : hyp 14.14'}
\end{align}
for all $(x,x',p,p',q,q',v)\in \R^{7}$. 
 
\noindent \bru{Assume for the moment that $q\mapsto F(x,p,q)$ is differentiable for all $(x,p)\in \R^{2}$.}
 For $n\ge 1$, existence of  a $C^{1,2}([0,T)\x \R)$ solution $\bar V_{n}$ to  \eqref{eq: PDE bar V} on $[0,T)\x (-n,n)$ with boundary condition $\bar V_{n}=0$ on $([0,T)\x \{-n,n\}) \cup (\{T\}\times [-n,n])$ follows from \cite[Theorem 14.24]{lieberman1996second}, \eqref{eq: F unif elliptic hyp 14.19 et 14.45}, \eqref{eq: hy 14.20} and \eqref{eq : hyp 14.14'}. \bru{It turns out that, using the notations of \cite[Theorem 14.24]{lieberman1996second},  $\bar V_{n}$ is even  in $H_{2+\theta_{B}}(B)$ for some $\theta_{B}\in (0,1)$, on each compact subset $B$ of $[0,T)\x (-n,n)$. These $H_{2+\theta_{B}}$-norms depend only on the upper and lower bounds on the derivative of $q\mapsto F(\cdot,q)$ and not on the fact that this map is differentiable. If it is not, one can thus first regularize $F$ with respect to its last argument, by using a sequence of smooth kernels, and then pass to the limit. The corresponding sequence will be uniformly bounded in $H_{2+\theta_{B}}(B)$ on each compact subset $B$ of $[0,T)\x (-n,n)$, so  that the limit will keep these bounds. By stability, the limit solves the required equation with the appropriate boundary conditions. See also the discussion is the paragraph preceding \cite[Theorem 14.24]{lieberman1996second}.}

\item We now provide uniform estimates on the gradients. Note that\bru{, by the Feynman-Kac formula and a comparison argument,} 
 \begin{align}\label{eq :def bar vn}  
\bar V_{n}(t,x)=\sup_{\bar\alpha \in \bar \Ac^{t}} \bar \E\left[\int_{t}^{T\wedge \tau^{t,x,\bar\alpha}_{n}}r(\bar X^{t,x, \bar\alpha}_{s}, \bar\alpha_{s}) \de s\right] 
\end{align}
where 
\[ \tau^{t,x,\bar\alpha}_{n}:=\inf\{s\ge t : \bar X^{t,x, \bar\alpha}_{s}\notin (-n,n)\}\,.\]
It follows that, for $h\in (0,T-t\bru{]}$,
 \begin{align*}
\bar V_{n}(t+h,x)=\sup_{\bar\alpha \in \bar \Ac^{t}} \bar \E\left[\int_{t}^{(T-h)\wedge \tau^{t,x,\bar\alpha}_{n}}r(\bar X^{t,x, \bar\alpha}_{s}, \bru{\bar \alpha_{s}}) \de s\right] 
\end{align*}
which readily implies that 
$
\abs{\bar V_{n}(t+h,x)-\bar V_{n}(t,x)}\le h \|r\|_{\infty},
$
and therefore 
\begin{align}\label{eq: borne partial t Vn}
\bru{\frac1T \norm{ \bar V_{n}} \vee} \norm{\partial_{t} \bar V_{n}}\le  \norm{r}_{\infty}.
\end{align}

Similarly, for $h\in (-1,1)$ such that $x+h\in [-n,n]$, 
\begin{align*}
\abs{\bar V_{n}(t,x+h)-\bar V_{n}(t,x)}
\le 
 \sup_{\bar\alpha \in \bar \Ac^{t}} \bar \E\left[\bru{\norm{r}_{{\rm Lip}}}\int_{t}^{T}\abs{\bar X^{t,x+h, \bar\alpha}_{s}-\bar X^{t,x, \bar\alpha}_{s}} \de s+\bru{\norm{r}_{\infty}}\abs{\tau^{t,x+h,\bar\alpha}_{n}-\tau^{t,x,\bar\alpha}_{n}}\right].
\end{align*}
The first term is handled by using the uniform Lipschitz continuity in space of $(\mu,\sigma)$: 
\begin{align}\label{eq: ecart X s + h}
 \bar \E\left[\int_{t}^{T}\abs{\bar X^{t,x+h, \bar \alpha}_{s}-\bar X^{t,x, \bar \alpha}_{s}} \de s\right] \le C_{1} \abs{h}
\end{align}
in which $C_{1}>0$ does not depend on $n$. 
As for the second term, Assumption \ref{asmp: Lipschitz diffusion coefs}, \bblue{\eqref{eq: hyp centrage et ellipticite b2}} and our boundedness assumptions \bblue{on $(b_{1},b_{2})$, and therefore on $(\mu,\sigma)$,} allow us to apply \cite[\bru{Theorem 2.3}]{bouchard2017first}\footnote{\bblue{Note that their Assumption (L) is not required since we are considering a finite time interval $[0,T]$, this can be easily seen from the proof of this Theorem.}} \bblue{with $\pi=0$, $r=1$  and}  for  $P$ of the form $\vp(\bar X^{t,x+h, \bar \alpha})$ or $\vp(\bar X^{t,x, \bar \alpha})$ for a smooth \bru{bounded} function $\vp$\bblue{, with bounded first and second derivatives,} such that $\vp(y)=y+n$ for $y\in [-n,\bblue{-n+1}]$ and $ \vp(y)=n-y$ for $y\in  [\bblue{n-1},n]$. It implies that 
\begin{align*} 
\bar \E\left[\abs{\tau^{t,x+h,\bar \alpha}_{n}-\tau^{t,x,\bar \alpha}_{n}}\right]&
\le 
C_{2}\l \bar\E\left[\abs{\bar X^{t,x+h, \bar \alpha}_{\tau^{t,x+h,\bar \alpha}_{n}\wedge \tau^{t,x,\bar \alpha}_{n}}-\bar X^{t,x, \bar \alpha}_{\tau^{t,x+h,\bar \alpha}_{n}\wedge \tau^{t,x,\bar \alpha}_{n}}}\right]  \le C'_{2} \abs{h}
\end{align*}
for some positive constants $C_{2}$ and $C'_{2}$ independent of $n$. 
Combined with \eqref{eq: ecart X s + h}, this leads to 
\begin{align}\label{eq: borne DVn}
\norm{\partial_{x} \bar V_{n}}_{\infty}\le \bru{\norm{r}_{{\rm Lip}}C_{1}+\norm{r}_{\infty}C'_{2}}\,,
\end{align}
\bru{in which, here and below, $\|\vp\|_{\infty}:=\sup\{|\vp(t,x)|: (t,x)\in [0,T)\x \R\}$ for a map $\vp:[0,T)\x\R\mapsto \R$.}
The fact that $\bar V_{n}$ solves  \eqref{eq: PDE bar V} combined with \eqref{eq: F unif elliptic hyp 14.19 et 14.45},   \eqref{eq: borne partial t Vn} and \eqref{eq: borne DVn} then proves that 
\begin{align}\label{eq : borne D2 bar Vn}
\norm{\partial^{2}_{xx} \bar V_{ n}}_{\infty}\le C_{3} 
\end{align}
for some $C_{3}>0$ that does not depend on $n$. 

\item We now prove the uniform H\"older continuity of the gradients and second derivatives. \bru{As in a) above, let us first assume that $F$ is $C^{1}$}. Given a neighbourhood ${\cal O}\subset [0,T]\x [-n,n] $ of a point $(t,x)$, we derive as in \cite[Section 3.1]{andrews2004fully} that 
there exists $C>0$ and $\beta \in(0,1]$, that depend only on   the ellipticity constant  $\eta$ and the Lipschitz constants of $F$ with respect to its second and third arguments, such that 
\[\abs{\partial_{t}\bar V_{n}(t',x')-\partial_{t}\bar V_{n}(t,x)}\le C \left(\abs{t'-t}^{\frac{\beta}{2}}+\abs{x'-x}^{\beta}\right)\sup_{\cal O}\norm{\partial_{t} \bar V_{n}}, \;\mbox{ for } (t',x')\in {\cal O}\,.\]
\bru{If $F$ is not $C^{1}$, one can first regularize it by using a sequence of kernels and then pass to the limit to obtain that the above still holds for the original $F$.}
In view of  \eqref{eq: borne partial t Vn}, this implies that
\begin{align}\label{eq: holder vt}
\abs{\partial_{t}\bar V_{n}(t',x')-\partial_{t}\bar V_{n}(t,x)}\le C \left(\abs{t'-t}^{\frac{\beta}{2}}+\abs{x'-x}^{\beta}\right) \norm{r}_{\infty}, \;\mbox{ for } (t',x')\in [0,T]\x \R.
\end{align}
Up to changing $\beta\in (0,1]$, one can prove similarly that 
\begin{align}\label{eq: holder vx}
\abs{\partial_{x}\bar V_{n}(t',x')-\partial_{x}\bar V_{n}(t,x)}\le C \left(\abs{t'-t}^{\frac{\beta}{2}}+\abs{x'-x}^{\beta}\right)  , \;\mbox{ for } (t',x')\in [0,T]\x \R,
\end{align}
for some $C>0$ that does not depend on $n$.
We now set $\Delta_{h} \bar V_{n}:= h^{-\beta}\left(\bar V_{n}(\cdot,\cdot+h)- \bar V_{n}\right)$, $h\in \R$. \bblue{Again, up to  mollifying $F$ with a smooth bounded kernel with derivatives bounded by $1$, we can assume that $F$ is $C^{1}$. Then, for $t<T$ and $x\in (-n+h,n-h)$, 
\begin{align*}
&h^{-\beta}\left\{F(x+h,\partial_{x}\bar V_{n}(t,x+h),\partial^{2}_{xx}\bar V_{n}(t,x+h) )-F(x,\partial_{x}\bar V_{n}(t,x),\partial^{2}_{xx}\bar V_{n}(t,x) )\right\}
\\
&=h^{-\beta}\left\{\partial_{x} F(x^{1}_{h},p^{1}_{h},q^{1}_{h})h+\partial_{p} F(x^{2}_{h},p^{2}_{h},q^{2}_{h})[\partial_{x}\bar V_{n}(t,x+h)-\partial_{x}\bar V_{n}(t,x)]\right.\\
&\;\;\;+\left.\partial_{q} F(x^{3}_{h},p^{3}_{h},q^{3}_{h})[\partial^{2}_{xx}\bar V_{n}(t,x+h)-\partial^{2}_{xx}\bar V_{n}(t,x)]\right\}
 \end{align*}
 for some $x^{i}_{h}\in [x,x+h]$, $p^{i}_{h}\in [\partial_{x}\bar V_{n}(t,x+h)\wedge \partial_{x}\bar V_{n}(t,x),\partial_{x}\bar V_{n}(t,x+h)\vee \partial_{x}\bar V_{n}(t,x)]$ and $q^{i}_{h}\in [\partial^{2}_{xx}\bar V_{n}(t,x+h)\wedge \partial^{2}_{xx}\bar V_{n}(t,x),\partial^{2}_{xx}\bar V_{n}(t,x+h)\vee \partial^{2}_{xx}\bar V_{n}(t,x)]$, for $i=1,2,3$.}
\bblue{It follows that $\Delta_{h} \bar V_{n}$} satisfies a linearized equation of the form
\begin{align*}
0=\partial_{t} \Delta_{h} \bar V_{n}+ A_{h} \partial_{x} \bru{(}\Delta_{h} \bar V_{n}\bru{)}+  B_{h} \partial^{2}_{xx} \bru{(}\Delta_{h} \bar V_{n}\bru{)} +C_{h} h^{1-\beta}
\end{align*}
at every point $(t,x)\in [0,T)\x \R$ such that $x+h\in (-n,n)$, 
in which, \bblue{by  Assumption \ref{asmp: Lipschitz diffusion coefs},  \eqref{eq: hyp centrage et ellipticite b2}} and the estimates in b) above, $(A_{h},C_{h})_{h>0}$ is uniformly bounded and $\inf_{h>0} \inf_{[0,T]\x \R }B_{h}\ge \eta/2\bru{>0}$.  Hence, 
\[\abs{\partial^{2}_{xx} \Delta_{h} \bar V_{n}}\le 2\eta^{-1}\left(\abs{\partial_{t} \Delta_{h} \bar V_{n}} +  \abs{A_{h}}  \abs{\partial_{x} \Delta_{h} \bar V_{n}}+ \abs{C_{h}} h^{1-\beta}\right)\]
We conclude from \eqref{eq: holder vt}-\eqref{eq: holder vx} that 
\begin{align}\label{eq: borne holder en espace D2Vn}
\abs{\partial^{2}_{xx} \bar V_{n}(t,x')-\partial^{2}_{xx}\bar V_{n}(t,x)}\le C\abs{x'-x}^{\beta},\; x,x'\in (-n,n),\;t<T\bru{,}
\end{align}
for some $C>0$ independent on $n$. 
If we now set $\Delta_{h} \bar V_{n}=h^{-\frac{\beta}{2}}(\bar V_{n}(\cdot+h,\cdot)-\bar V_{n})$, then the same type of arguments leads to 
\begin{align}\label{eq: borne holder en temps D2Vn}
\abs{\partial^{2}_{xx} \bar V_{n}(t',x)-\partial^{2}_{xx}\bar V_{n}(t,x)}\le C\abs{t'-t}^{\frac{\beta}{2}},\; x  \in (-n,n),\;t,t'<T\bru{,}
\end{align}
for some $C>0$ independent on $n$. 

\item  It follows from steps b)~and c)~that $(\bar V_{n})_{n\ge 1}$ is uniformly bounded in $H_{2+\beta}([0,T) \x \R)$, as defined in \cite[Section IV.1]{lieberman1996second}.  By the Arzel\`a-Ascoli theorem, it admits a subsequence that converges  in $H_{2+\beta}(B)$, for any compact set $B\subset [0,T)\x \R$\bru{, to a limit $\bar V_{\infty}$}.  \bru{This limit shares the same  upper-bound in $H_{2+\beta} ([0,T)\x \R)$ as $(\bar V_{n})_{n\ge 1}$.} \bru{Since each $\bar V_{n}$ solves \eqref{eq: PDE bar V} on $[0,T)\times (-n,n)$ and satisfies the boundary condition \eqref{eq: cond T bar V} on $[-n,n]$, it follows that $\bar V_{\infty}$ solves \eqref{eq: PDE bar V} on $[0,T)\times \R$ and \eqref{eq: cond T bar V} on $\R$. As $\bar V$ is also a bounded solution of the same equation, comparison implies that $\bar V_{\infty}=\bar V$. 
 }
 \end{enumerate}
\end{proof}

\begin{Remark}\label{rem : sur choix beta}
\begin{enumerate}[label=(\rm\alph*)]
\item Let $\bar \ar : [0,T)\x \R \mapsto \Ab$ be a measurable map satisfying 
\begin{align}\label{eq: def bar a theta} 
\bar \ar\in \underset{a\in \Ab}{\rm argmax}\left(\mu(\cdot,a)\partial_{x} \bar V  +\frac12 \sigma^{2}(\cdot,a) \partial^{2}_{xx}\bar V +r(\cdot,a)\right)\;\mbox{ on } [0,T)\x \R,
\end{align}
see e.g.~\cite[Proposition 7.33, p.153]{BertsekasShreve.78}. Assume that there exists $\beta_{\circ}\in (0,1)$ such that \bblue{$(\mu,\sigma,r)(\cdot,\bar \ar)$} belongs to $H_{\beta_{\circ}}([0,T)\x \R)$, then we can take $\beta=\beta_{\circ}$. This follows from \cite[Section IV.14, p390]{ladyzhenska1988linear}.
\item If $(\mu(\cdot,\bar \ar),\sigma(\cdot,\bar \ar),r(\cdot,\bar \ar))$ has more regularity, one can obviously obtain more regularity on $\bar V$ by, for instance, differentiating the associated partial differential equation. 
\item In the case where  $\sigma$ does not depend on its $a$-argument, then one can appeal to  \cite[Theorem 12.16]{lieberman1996second} to deduce that   we can take $\beta=1$. This follows from the  Lipschitz continuity of $F$.
\end{enumerate}
\end{Remark}

\subsection{Convergence speed toward the diffusive limit}\label{subsec : conv speed}

We now exploit the H\"older regularity stated above  to prove that   $V_{\epsilon}$ converges to $\bar V$ at a rate $\epsilon^{\frac{\beta}{2}}$ as $\epsilon$ vanishes. We shall see in Section \ref{subsec: approx ctrl} below that it provides an $\epsilon^{\frac{\beta}{2}}$-optimal control for the pure-jump problem.  In general, it can not be improved, see Example \ref{example : first correction} in Section \ref{subsec: first order correction} below.

\begin{Theorem}\label{thm : ecart V et bar V} For all $(t,x)\in [0,T]\x \R$ and $\epsilon>0$, 
\begin{align*}
\abs{V_{\epsilon}-\bar V}(t,x)&\le\ \sup_{\alpha\in \Ac^{t}}  \E\left[\bblue{\int_{t}^{T} |\delta r_{\epsilon}|(X^{t,x,\alpha}_{s},\alpha_{s})  \de s}\right]
\end{align*}
in which 
\begin{align}\label{eq: def delta r eps}
\delta r_{\epsilon}:=
\epsilon^{-1}\bru{\int \left(\bar V(\cdot,\cdot+b_{\epsilon})-\bar V\right)\nu(\de e)}- \bru{\mu} \partial_{x} \bar V-\frac12 \bru{\sigma^{2}}\partial^{2}_{xx} \bar V
\end{align}
satisfies
\begin{align}\label{eq: borne delta r epsilon dans thm} 
\norm{\delta r_{\epsilon}}_{\infty}\le C^{\epsilon}_{K}\epsilon^{\frac{\beta}{2}}
\end{align}
 with  
\begin{align*}
C^{\epsilon}_{K}:=&\bblue{ \frac12  \norm{\partial^{2}_{xx} \bar V}_{\infty}(\epsilon^{1-\frac{\beta}{2}} \|b_1\|^{2}_{\infty} +2\epsilon^{\frac{1-\beta}2} \|b_1\|_{\infty} \|b_2\|_{\infty})+\frac{K}{2}(\epsilon^{\frac12} \|b_1\|_{\infty} +\|b_2\|_{\infty})^{2+\beta}},
\end{align*}
where $K>0$ is the H\"older constant of $\partial^{2}_{xx}\bar V$ with respect to its space variable.

\noindent In particular, 
\[\limsup_{\epsilon \downarrow 0} \epsilon^{-\frac{\beta}{2}} \norm{V_{\epsilon}(t,\cdot)-\bar V(t,\cdot)}_{\infty}\le \frac12 (T-t) K( \|b_2\|_{\infty})^{2+\beta}\,,\; t\bru{\le}T.\] 
\end{Theorem} 

\begin{proof} Since  $\bar V\in C^{1,2}_{b}([0,T\bru{)}\x \R)$,
\bblue{
\begin{align*}
&\bar V(t,x+b_{\epsilon}(x,a,e))-\bar V(t,x)\\
&=\partial_{x}\bar V(t,x)b_{\epsilon}(x,a,e)+\frac12\partial^{2}_{xx}\bar V(t,x)|b_{\epsilon}(x,a,e)|^{2}
+\frac12(\partial^{2}_{xx}\bar V(t,x_{\epsilon})-\partial^{2}_{xx}\bar V(t,x))|b_{\epsilon}(x,a,e)|^{2}
 \end{align*}
 for some $x_{\eps}$ that lies in the interval formed by $x$ and $x+b_{\epsilon}(x,a,e)$. By the left-hand side of \eqref{eq: hyp centrage et ellipticite b2}, the definition of $(\mu,\sigma)$, and since $\partial^{2}_{xx} \bar V$ is $\beta$-H\"older continuous in space with constant $K$,
\begin{align*}
&\left|\frac1\epsilon \int \left(\bar V(t,x+b_{\epsilon}(x,a,e))-\bar V(t,x)\right)\nu(\de e)- \mu(x,a) \partial_{x}\bar V(t,x)-\frac12\sigma^{2}(x,a)\partial^{2}_{xx}\bar V(t,x)\right|
\\
&\le \frac12 \norm{\partial^{2}_{xx} \bar V}_{\infty}(\epsilon \|b_1\|^{2}_{\infty} +2\epsilon^{\frac12} \|b_1\|_{\infty} \|b_2\|_{\infty})+\epsilon^{\frac{\beta}{2}}\frac{K}{2}(\epsilon^{\frac12} \|b_1\|_{\infty} +\|b_2\|_{\infty})^{2+\beta}
\end{align*}
Hence, 
\begin{align}\label{eq: approx edp}
\mu\partial_{x} \bar V +\frac12 \sigma^{2} \partial^{2}_{xx}\bar V+r
&=\frac1\epsilon \int \left(\bar V(\cdot,\cdot+b_{\epsilon}(\cdot,e))-\bar V(t,x)+\epsilon (r-\delta r_{\epsilon})\right)\nu(\de e)
\end{align} 
where $\delta r_{\epsilon}$ is the continuous function, defined in \eqref{eq: def delta r eps}, and satisfies
\begin{align*}
\abs{\delta r_{\epsilon}}\le& \frac12  \norm{\partial^{2}_{xx} \bar V}_{\infty}(\epsilon \|b_1\|^{2}_{\infty} +2\epsilon^{\frac12} \|b_1\|_{\infty} \|b_2\|_{\infty})+\epsilon^{\frac{\beta}{2}}\frac{K}{2}(\epsilon^{\frac12} \|b_1\|_{\infty} +\|b_2\|_{\infty})^{2+\beta}.
\end{align*}}
Combined with Proposition \ref{prop: HJB bar V}, this shows that $\bar V$ is a smooth solution of 
\begin{align}\label{eq: approx edp 2}
\begin{cases}\displaystyle \partial_{t} \bar V+\sup_{a\in \Ab} \frac1\epsilon  \int \left( \bar V(\cdot,\cdot+b_{\epsilon}(\cdot,a,e))-\bar V+\epsilon(r(\cdot,a)-\bblue{\delta r_{\epsilon}(\cdot,a))}\right)\nu(\de e)=0,\; \mbox{on}\; [0,T)\x \R,\\
\displaystyle \bar V(T,\cdot)=0,\; \mbox{ on } \R.
\end{cases}
\end{align}
Applying Proposition \ref{prop: HJB saut + continuite} (with the appropriate coefficients), this implies that 
\begin{align*}
\bar V(t,x)&=\sup_{\alpha\in \Ac^{t}}  \E\left[\int_{t}^{T}  \int\epsilon ( r- \delta r_{\epsilon})\bblue{(X^{t,x,\alpha}_{s-},\alpha_{s}) }N(\de e,\de s)\right],
\end{align*}
so that, by the definition of $V_\epsilon$, 
\begin{align*}
\abs{V_{\epsilon}-\bar V}(t,x)&\le \sup_{\alpha\in \Ac^{t}}  \E\left[\int_{t}^{T}  \int \epsilon \abs{\delta r_{\epsilon}}\bblue{(X^{t,x,\alpha}_{s-},\alpha_{s}) } N(\de e,\de s)\right]\\
&= \sup_{\alpha\in \Ac^{t}}  \E\left[\int_{t}^{T}      \abs{\delta r_{\epsilon}}\bblue{(X^{t,x,\alpha}_{s},\alpha_{s}) }  \de s\right].
\end{align*}
\end{proof}

\subsection{Construction of a $\epsilon^{\frac{\beta}{2}}$-optimal control for the pure-jump problem}\label{subsec: approx ctrl}

We now show that  an $\epsilon^{\frac{\beta}{2}}$-optimal control for  \eqref{def : hat V eps} can be constructed by considering a measurable map $\bar \ar: [0,T)\x \R \mapsto \Ab$  satisfying 
\begin{align}\label{eq: def bar a theta}
\bar \ar\in \underset{\bar a\in \Ab}{\rm argmax}\left(\mu(\cdot,\bar a)\partial_{x} \bar V +\frac12 \sigma^{2}(\cdot,\bar a) \partial^{2}_{xx}\bar V +r(\cdot,\bar a)\right)\;\mbox{ on } [0,T)\x \R,
\end{align}
see e.g.~\cite[Proposition 7.33, p.153]{BertsekasShreve.78}, and define $\bar \alpha^{t,x} \in \Ac^{t}$ by 
\begin{align*}
\bar \alpha^{t,x}_{s}=\bar \ar(s,X^{t,x,\bar \alpha^{t,x}}_{s-}), \;s\in [t,T\bru{)},
\end{align*}
recall  \eqref{def : X}. As it is driven by a compound Poisson process,  the couple $(X^{t,x,\bar \alpha^{t,x}},\bar \alpha^{t,x})$ is well-defined.

\begin{Proposition}\label{prop: construction control eps opti} For all $(t,x)\in [0,T\bru{)}\x \R$ and $\epsilon>0$, $\bar \alpha^{t,x}$ is $\epsilon^{\frac\beta 2}$-optimal for $V_{\epsilon}$. Namely, 
\begin{align*}
\frac1{\lambda_{\epsilon}} \E\left[\int_{t}^{T}r(X^{t,x,\bar \alpha^{t,x}}_{s-},\bar \alpha^{t,x}_{s}) \de N_{s}\right]
 \ge  V_{\epsilon}(t,x)-2(T-t)C^{\epsilon}_{K}\epsilon^{\frac{\beta}{2}}\bru{.}
\end{align*}

\end{Proposition}

\begin{proof}
It follows from Proposition \ref{prop: HJB bar V}, \eqref{eq: def bar a theta} and   \eqref{eq: approx edp} that 
\begin{align*}
&\partial_{t} \bar V+\frac1\epsilon  \int \left( \bar V(\cdot,\cdot+b_{\epsilon}(\cdot,\bar \ar,e))-\bar V+\epsilon r(\cdot,\bar \ar)\right)\nu(\de e)\ge -\norm{\delta r_{\epsilon}}_{\infty},\;  \\
&\partial_{t} \bar V+\sup_{a\in \Ab}\frac1\epsilon  \int \left( \bar V(\cdot,\cdot+b_{\epsilon}(\cdot,a,e))-\bar V+\epsilon r(\cdot,a)\right)\nu(\de e)\le \norm{\delta r_{\epsilon}}_{\infty},   
\end{align*}
so that applying \bru{It\^{o}'s Lemma and using \eqref{eq: cond T bar V} leads} to 
\begin{align*}
 \bar V(t,x) - (T-t) \norm{\delta r_{\epsilon}}_{\infty} &\le \frac1{\lambda_{\epsilon}} \E\left[\int_{t}^{T} r(X^{t,x,\bar \alpha^{t,x}}_{s-},\bar \alpha^{t,x}_{s})\de N_{s}\right]\\
\bar V(t,x) + (T-t) \norm{\delta r_{\epsilon}}_{\infty} &\ge \sup_{\alpha \in \Ac^{t}} \frac1{\lambda_{\epsilon}}
\E \left[\int_{t}^{T}    r(X^{t,x,  \alpha}_{s-},  \alpha_{s}) \de N_{s}\right]=V_{\epsilon}(t,x).
\end{align*}
We conclude by appealing to \eqref{eq: borne delta r epsilon dans thm}. 
 \end{proof}


\subsection{First order correction term}\label{subsec: first order correction}

Under additional conditions, one can exhibit a first order correction term to improve the convergence speed in Theorem \ref{thm : ecart V et bar V} and Proposition \ref{prop: construction control eps opti}. From now on, we  assume the following.
\begin{Assumption}\label{ass: first order expansion}\hfill\vspace{-1em}
\begin{enumerate}[label=\rm\alph*.] 
\item The map $(t,x,a)\in [0,T)\x \R\x \Ab \mapsto \epsilon^{-\frac{\beta}{2}} \bblue{\delta r_{\epsilon}(t,x,a)}$ is continuous, uniformly in $\bblue{\epsilon \in (0,1)}$.
\item The pointwise limit
\begin{align}
 r_{1} &:=\lim_{\epsilon \to 0}    \bblue{ \epsilon^{-\frac{\beta}{2}} \delta r_{\epsilon}},\label{eq :def delta r i}
\end{align}
is well-defined on $[0,T)\x \R$.
\item Given 
$$
 \Ab_{0}:= \underset{\bar a\in \Ab}{\rm argmax}\left(\mu(\cdot,\bar a)\partial_{x} \bar V  +\frac12 \sigma^{2}(\cdot,\bar a) \partial^{2}_{xx}\bar V +r(\cdot,\bar a)\right),
 $$
  comparison holds in the sense of bounded discontinuous viscosity super- and subsolutions for 
\begin{align}
\begin{cases}\displaystyle\partial_{t}\vp + \max_{\bar a \in \Ab_{0}}  \left( \mu(\cdot,\bar a) \partial_{x}\vp+\frac12 \sigma(\cdot,\bar a)^{2} \partial^{2}_{xx}\vp+r_{1}(\cdot,\bar a)\right) =0,\; \mbox{ on } [0,T)\times \R\\
\displaystyle\vp(T,\cdot)=0\; \mbox{ on } \R\end{cases}\,\label{eq: PDE W i}
\end{align}
 and \eqref{eq: PDE W i} admits a (unique) bounded viscosity solution, denoted by $\delta \bar V^{(1)}$.
\end{enumerate}

\end{Assumption}

\begin{Remark} Let us comment the above: 
\vspace{-1em}
\begin{itemize}
\item[{\rm a)}] Note that $ r_{1}$ is bounded, see \eqref{eq: borne delta r epsilon dans thm} in Theorem \ref{thm : ecart V et bar V}.  The right-hand side term in \eqref{eq :def delta r i} therefore admits a limsup and a liminf. 
The condition \eqref{eq :def delta r i} implies that the limit is actually well-defined. This point will be further discussed in Remark \ref{rem : si r1 pas def} below.
\item[{\rm b)}] If $\bar V$ admits a continuous bounded third-order space derivative $\partial^{3}_{xxx}\bar V$, then one easily checks that $\beta=1$ and $r_{1}=\frac{1}{6}\int |b_{2}(\cdot,e)|^{3}\nu(\de e) \partial^{3}_{xxx}\bar V$, by a simple Taylor expansion.  
\item[{\rm c)}] \bblue{Assume that one can find a continuous map $\bar {\rm a}:[0,T)\x \R \mapsto \Ab$ such that $\bar{\rm a}(t,x)\in \Ab_{0}(t,x)$ for all $(t,x)\in [0,T)\x \R$, and 
$x\in \R\mapsto (\mu,\sigma)(x,\bar{\rm a}(t,x))$ is  Lipschitz uniformly in $t\le t_{0}$, for all $t_{0}<T$. Also assume that $r_{1}$ is continuous, then comparison holds, see e.g.~\cite[Section 8]{CrandallIshiiLions}. 
In general, this} can be checked on a case-by-case basis. 
\end{itemize}
\end{Remark}
 
Under the above conditions, $\delta \bar V^{(1)}$ is the  first order term in the difference $V_{\epsilon}-\bar V$, i.e.~\eqref{eq: conv speed V eps - bar V1 eps} below holds with
\begin{align}
\bar V^{(1)}_{\epsilon}&:= \bar V+ \epsilon^{\frac{\beta}{2}} \delta \bar V^{(1)}. \label{eq: def bar V i}
\end{align} 

\begin{Theorem}\label{Thm : first correction term}  Let Assumption \ref{ass: first order expansion} hold. Then, for all $(t,x)\in [0,T]\x \R$, 
\[
\lim_{\epsilon\downarrow 0}\epsilon^{-\frac{\beta}{2}}(V_{\epsilon}-\bar V)(t,x)=\delta \bar V^{(1)}(t,x)
\]
and therefore 
\begin{align}\label{eq: conv speed V eps - bar V1 eps}
\limsup_{\epsilon \downarrow 0} \epsilon^{-\frac{\beta}{2}}  \abs{  V_{\epsilon}(t,x) - \bar V^{(1)}_{\epsilon} (t,x)    } =0.
\end{align}
If  in addition $\delta \bar V^{(1)}$ is $C^{1,2}([0,T)\x \R)$ and $\partial^{2}_{xx}\delta \bar V^{(1)}$ is $\delta \beta$-H\"older continuous in space\bru{, uniformly on $[0,T)\x \R$,} for some constant $\delta \beta>0$ such that 
\begin{align}\label{eq: vitesse delta eps}
 \limsup_{\epsilon\downarrow 0} \epsilon^{-\frac{\delta \beta}{2}}\norm{\bblue{\epsilon^{-\frac{  \beta}{2}}\delta r_{\epsilon}-r_{1}}}_{\infty}<\infty,
\end{align}
 then the control defined by
\begin{align}\label{eq: def check alpha}
\check \alpha^{t,x}_{s}=\check \ar(s,X^{t,x,\check \alpha^{t,x}}_{s-}), \;s\in [t,T\bru{)}
\end{align}
with 
\begin{align}\label{eq: def a check}
\check \ar\in \underset{\bar a\in \Ab_{0}}{\rm argmax}\left\{ \mu(\cdot,\bar a) \partial_{x}\delta \bar V^{(1)}+\frac12 \sigma(\cdot,\bar a)^{2} \partial^{2}_{xx}\delta \bar V^{(1)} +r_{1}(\cdot,\bar a) \right\},\;\mbox{ on }   [0,T)\x \R,
\end{align}
satisfies 
\begin{align*} 
\frac1{\lambda_{\epsilon}} \E\left[\int_{t}^{T}      r(X^{t,x,\check \alpha^{t,x}}_{s-},\check \alpha^{t,x}_{s}) \de N_{s}\right]
 \ge  V_{\epsilon}(t,x)-\bblue{o(\epsilon^{\frac{\beta}{2} })},\;\mbox{ for all $\epsilon>0$}, 
\end{align*}
 \bru{where $o$ is a continuous bounded function such that $o(y)/y\to 0$ as $y\downarrow 0$.}
\end{Theorem}

\begin{proof}\ We split the proof in two steps.
\\
a. Let us set  $W_{\epsilon}:=\epsilon^{-\frac{\beta}{2}}(V_{\epsilon}-\bar V)$ and consider its relaxed semi-limits
\[
W^{*}(t,x):=\limsup_{\substack{(t',x')\to (t,x)\\\epsilon\downarrow 0}} W_{\epsilon}(t',x'),\;W_{*}(t,x):=\liminf_{\substack{(t',x')\to (t,x)\\\epsilon\downarrow 0} } W_{\epsilon}(t',x')\,.
\]
Note that Theorem \ref{thm : ecart V et bar V} ensures that the above are well-defined and  bounded.
We claim that $W^{*}$ and $W_{*}$ are respectively  bounded sub- and supersolutions of \eqref{eq: PDE W i}.
For brevity, we will only include the \bblue{details for the} proof of the subsolution property, the supersolution property is proved similarly \bblue{and we only mention how to adapt the arguments}.
Fix $\vp\in C^{1,2}_{b}$ and let $(t_{\circ},x_{\circ})\in [0,T)\x \R$ achieve a maximum of $W^{*}-\vp$ on a ball $B_{k}:=\{(t,x) \in [0,T)\x \R : |t_{\circ}-t'|\le (T-t_{\circ})/2, \; |x_{\circ}-x'| \le k\}\subset [0,T)\x \R$, for some $k>0$. 
Then, there exist a sequence $(t_{\epsilon_{n}},x_{\epsilon_{n}})_{\epsilon_{n}}$ such that $\epsilon_{n}\to 0$, $W_{{\epsilon_n}}(t_{\epsilon_{n}},x_{\epsilon_{n}})\to W^{*}(t_{\circ},x_{\circ})$, \bblue{ $(t_{\epsilon_{n}},x_{\epsilon_{n}})\to (t_{\circ},x_{\circ})$}, and such that $(t_{\epsilon_{n}},x_{\epsilon_{n}})$ is a   maximum of $
W_{\epsilon_{n}}-\vp$ in the interior of $B_{2k}$, see e.g.~\cite[Lemma 6.1]{Bar94}. For $k> \epsilon_{n}^{\frac12}  (\|b_{1}\|_{\infty}+\|b_{2}\|_{\infty})$, the viscosity subsolution property of $ V_{\epsilon_{n}}$, applying  Proposition \ref{prop: HJB saut + continuite} to the test function $\bar{V}+ {\epsilon_{n}^{\frac{\beta}{2}}}\varphi$, implies that  
\begin{align*}
0\le & \partial_{t}(\bar V+\epsilon_{n}^{\frac{\beta}{2}}\vp)(t_{\epsilon_{n}},x_{\epsilon_{n}}) \\
&+  \frac1{\epsilon_{n}}\left(\int \left(\bar V+\epsilon_{n}^{\frac{\beta}{2}}\vp\right)(t_{\epsilon_{n}},x_{\epsilon_{n}}+b_{\bblue{\epsilon_{n}}}(x_{\epsilon_{n}},\bar a_{{n}},e))\nu(\de e)- \left(\bar V+\epsilon_{n}^{\frac{\beta}{2}}\vp\right)(t_{\epsilon_{n}},x_{\epsilon_{n}})+\epsilon_{n} r(x_{\epsilon_{n}},\bblue{\bar a_{{n}}}) \right) 
\end{align*}
for some $\bar a_{{n}}\in \Ab$. \bblue{Since $\vp\in C^{1,2}_{b}$, a second order Taylor expansion combined with Assumption \ref{asmp: b2 conds} implies that 
\begin{align*}
\lim_{n\to \infty} \epsilon_{n}^{\frac{\beta}{2}} \left[\partial_{t}\vp(t_{\epsilon_{n}},x_{\epsilon_{n}}) +  \frac1{\epsilon_{n}}\left(\int \vp(t_{\epsilon_{n}},x_{\epsilon_{n}}+b_{\bblue{\epsilon_{n}}}(x_{\epsilon_{n}},\bar a_{{n}},e))\nu(\de e)- \vp(t_{\epsilon_{n}},x_{\epsilon_{n}}) \right) 
\right]=0.
 \end{align*}
 Thus, if $\bar a$ is a limit point of $(\bar a_{n})_{n\ge 1}$, we deduce from \eqref{eq: def delta r eps}-\eqref{eq: borne delta r epsilon dans thm} and the above that 
 $$
 0\le \partial_{t}\bar V(t_{\circ},x_{\circ})+\mu(x_{\circ},\bar a)\partial_{x} \bar V(t_{\circ},x_{\circ}) +\frac12 \sigma^{2}(x_{\circ},\bar a) \partial^{2}_{xx}\bar V(t_{\circ},x_{\circ}) +r(x_{\circ},\bar a).
 $$
In view of Proposition \ref{prop: HJB bar V}, this shows that $\bar a_{{n}}$ converges to some element of $\bar a\in  \Ab_{0}(t_{\circ},x_{\circ})$ as $n$ goes to infinity,  after possibly  passing to a subsequence.} \bru{By  \eqref{eq: approx edp 2} and the above,}   
\[
0\le \partial_{t} \vp(t_{\epsilon_{n}},x_{\epsilon_{n}}) +   \frac1{\epsilon_{n}}\int\left(  \vp(t_{\epsilon_{n}},x_{\epsilon_{n}}+b_{\bblue{\epsilon_{n}}}(x_{\epsilon_{n}},\bar a_{{n}},e))- \vp(t_{\epsilon_{n}},x_{\epsilon_{n}})  +\epsilon_{n} \epsilon_{n}^{-\frac{\beta}{2}}\bblue{\delta r_{\epsilon_{n}}(x_{\epsilon_{n}},\bar a_{n})}\right)\nu(\de e).
\]
Sending $n\to \infty$ and using parts a.~and b.~of Assumption \ref{ass: first order expansion} together with Assumption \ref{asmp: b2 conds},  this leads to
\[
0\le \partial_{t} \vp(t_{\circ},x_{\circ}) +  \mu(x_{\circ},\bar a) \partial_{x}\vp(t_{\circ},x_{\circ})+\frac12 \sigma(x_{\circ},\bar a) \partial^{2}_{xx}\vp(t_{\circ},x_{\circ}) +r_{1}(x_{\circ},\bar a),
\]
so that the required subsolution property is proved on $ [0,T)\times \R$. The fact that $W^{*}(T,\cdot )\le 0$ follows from the last assertion of Theorem \ref{thm : ecart V et bar V}. 
\\
\bblue{To prove the supersolution property, it suffices to follow the same arguments but choose $\bar a_{n}\in \Ab_{0}(t_{\epsilon_{n}},x_{\epsilon_{n}})$. For a test function $\vp\in C^{1,2}_{b}$ for $W_{*}$ at $(t_{\circ},x_{\circ})\in [0,T)\x \R$, keeping the same notations as above, this lead to 
\begin{align*}
 0\ge & \partial_{t}(\bar V+\epsilon_{n}^{\frac{\beta}{2}}\vp)(t_{\epsilon_{n}},x_{\epsilon_{n}}) \\
&+   \frac1{\epsilon_{n}}\left\{\int \left(\bar V+\epsilon_{n}^{\frac{\beta}{2}}\vp\right)(t_{\epsilon_{n}},x_{\epsilon_{n}}+b_{\bblue{\epsilon_{n}}}(x_{\epsilon_{n}},\bar a_{{n}},e))\nu(\de e)- \left(\bar V+\epsilon_{n}^{\frac{\beta}{2}}\vp\right)(t_{\epsilon_{n}},x_{\epsilon_{n}})+\epsilon_{n} r(x_{\epsilon_{n}},\bblue{\bar a_{{n}}}) \right\} \\
=& \epsilon_{n}^{\frac{\beta}{2}}\left( \partial_{t} \vp(t_{\epsilon_{n}},x_{\epsilon_{n}}) +   \frac1{\epsilon_{n}}\int\left(  \vp(t_{\epsilon_{n}},x_{\epsilon_{n}}+b_{\bblue{\epsilon_{n}}}(x_{\epsilon_{n}},\bar a_{{n}},e))- \vp(t_{\epsilon_{n}},x_{\epsilon_{n}})  +\epsilon_{n} \epsilon_{n}^{-\frac{\beta}{2}}\bblue{\delta r_{\epsilon_{n}}(x_{\epsilon_{n}},\bar a_{n})}\right)\nu(\de e)\right)
\end{align*}
by Proposition \ref{prop: HJB bar V} and \eqref{eq: def delta r eps}.} 
\\
By comparison,   $W:= W^{*}=W_{*}$ is the unique bounded viscosity solution of \eqref{eq: PDE W i} and is therefore equal to $\delta \bar V^{(1)}$.

b. We now assume that $\delta \bar V^{(1)}$ is $C^{1,2}([0,T)\x \R)$ and that $\partial^{2}_{xx}\delta \bar V^{(1)}$ is $\delta \beta$-H\"older continuous in space\bru{, uniformly on $[0,T)\x \R$,} for some $\delta \beta>0$ such that 
\eqref{eq: vitesse delta eps} holds. Using \eqref{eq: vitesse delta eps}  and the same arguments as in the proof of Theorem \ref{thm : ecart V et bar V} lead to
\begin{align}\label{eq: vitesse delta r1 eps}
 \limsup_{\epsilon\downarrow 0} \epsilon^{-\frac{\delta \beta}{2}} \bblue{\norm{  \delta r^{(1)}_{\epsilon}}_{\infty}}<\infty,
\end{align}
in which 
\begin{align*}
 \delta r^{(1)}_{\epsilon} &:= \frac1\epsilon \bblue{\int\left(\delta \bar V^{(1)}(\cdot,\cdot+b_{\epsilon})-\delta \bar V^{(1)}\right) \nu(\de e)}
+ \epsilon^{-\frac{\beta}{2}}\delta r_{\epsilon} -  \bru{\mu}\partial_{x}\delta \bar V^{(1)}-\frac12 \bru{\sigma}^{2} \partial^{2}_{xx}\delta \bar V^{(1)} -r_{1}.
\end{align*}
Moreover, direct computations using the above and \eqref{eq: approx edp 2} show that $\bar V^{(1)}_{\epsilon}$ defined in \eqref{eq: def bar V i}  solves 
\begin{align*}
\partial_{t}\bar V^{(1)}_{\epsilon}+  \frac1\epsilon \int \left(\bar V^{(1)}_{\epsilon}(\cdot,\cdot+b_{\epsilon}(\cdot,\check a,e))-\bar V^{(1)}_{\epsilon}(t,x)-\epsilon \epsilon^{\frac{\beta}{2}}\bblue{\delta r^{(1)}_{\epsilon}(\cdot,\check a)}\right)\nu(\de e)+r(\cdot,\check a)=0
\end{align*}
on $[0,T)\x \R$, where $\check a$ is defined as in \eqref{eq: def a check}. Together with  \eqref{eq: vitesse delta r1 eps}, this implies that, for $\check \alpha^{t,x}$ defined as in \eqref{eq: def check alpha}, we have 
\begin{align*}
\frac1{\lambda_{\epsilon}} \E\left[\int_{t}^{T}      r(X^{t,x,\check \alpha^{t,x}}_{s-},\check \alpha^{t,x}_{s}) \de N_{s}\right]
 \ge  \bar V^{(1)}_{\epsilon}(t,x)-\epsilon^{\frac{\beta}{2}}O(\epsilon),
\end{align*}
in which $O$ is a continuous function with $O(0)=0$. On the other hand, it follows from Step a.~that $|V_{\epsilon}(t,x)-\bar V^{(1)}_{\epsilon}(t,x)|\le o(\epsilon^{\frac{\beta}{2}})$.
\end{proof}

\begin{Remark}\label{rem : si r1 pas def} If the limit in \eqref{eq :def delta r i} is not defined, one can still define its \bru{relaxed} limsup \bru{and  liminf} (recall that it is bounded). Let us denote them by $r_1^{*}$ and $r_{1*}$ respectively. Then, 
$W^{*}$ defined in the above proof is simply a viscosity sub-solution of \eqref{eq: PDE W i} with $r_1^{*}$ in place of $r_{1}$. Similarly, $W_{*}$ is a viscosity super-solution of the same equation but with $r_{1*}$ in place of $r_{1}$. This still provides asymptotic upper- and lower-bounds for $\epsilon^{-\frac{\beta}{2}}(V_{\epsilon}-\bar V)$.

\end{Remark}

\begin{Example}\label{example : first correction} To illustrate the above, we consider a toy model in which explicit solutions can be derived. Although it does not satisfy our general assumptions, e.g.~of boundedness and H\"older regularity in space, we shall see that a similar approach can still be applied. We consider the dynamics
$$
X^{t,x,\alpha}=x+ \int_{t}^{\cdot } X^{t,x,\alpha}_{s-} \int (\epsilon b_{1} (\alpha_{s},e)+\sqrt{\epsilon} b_{2}(\alpha_{s},e) )N(\de e,\de s),
$$
in which $b_{1}$ and $b_{2}$ are bounded and continuous with respect to their first argument, uniformly in the second one. 
For $\gamma \in (0,1]$, the value function is defined as 
$$
V_{\epsilon}(t,x)=\sup_{\alpha\in \Ac^{t}} \frac1{\lambda_{\epsilon}}\E\left[\int_{t}^{T} \int |X^{t,x,\alpha}_{s-}|^{\gamma} r(\alpha_{s}) \de N_{s}\right],
$$
for some continuous function $r$.  Then, one easily checks that $\bar V(t,x)= \bar  f(t)|x|^{\gamma}$ in which $\bar  f$ solves 
\begin{align*}
\partial_{t} \bar f  +   \sup_{\bar a\in \Ab} \left( \bar f\{\gamma \mu (\bar a)+\frac12 \gamma(\gamma-1) \sigma^{2}(\bar a)\}  +r(\bar a) \right) =0, \mbox{ on } [0,T)\x \R,
\end{align*}
with $\bar f(T)=0$. Because $|x|^{\gamma}$ factorizes, the H\"older constant of $\partial^{2}_{xx}\bar V$ can be considered around $x=1$. Since the third-order space derivative of $\bar V$ is bounded in a neighbourhood of $1$,  Theorem \ref{thm : ecart V et bar V} applies with  $\beta=1$. The convergence rate is therefore of order $\epsilon^{\frac12}$. Moreover, by direct computations, the first order correction term is of the form  $\delta \bar V^{(1)}(t,x)=\delta \bar f(t)|x|^{\gamma}$ where $\delta \bar f\not\equiv 0$ solves 
$$
\partial_{t}\delta\bar f + \sup_{\bar a\in \Ab_0}\left( \delta \bar f\{\gamma \mu (\bar a)+\frac12 \gamma(\gamma-1) \sigma^{2}(\bar a)\}  +r_{1}(\cdot,\bar a) \right)=0 
$$
with  $\delta \bar f(T)=0$, in which 
$$
(t,\bar a)\in [0,T]\x \Ab \mapsto r_{1}(t,\bar a):= \gamma(\gamma-1)\ell\left(  \int (b_1b_{2})(\bar a,e)\nu(\de e)\right) \bar f(t)
$$
for some  (explicit) continuous map  $\ell$ with linear growth. In particular, this shows that the convergence rate in $\epsilon^{\frac12}$ proved in Theorem \ref{thm : ecart V et bar V} is sharp.
\end{Example}

\subsection{Higher order expansions}\label{subsec: higher order expansions}

To conclude this section, note that higher order expansions can be obtained, upon existence of an associated systems of parabolic equations. Namely, let   us assume the following.

\begin{Assumption}\label{ass: existence systeme} There exists $(\delta \beta_{i})_{i=0,\cdots,i_{\circ}}\subset (0,1]^{i_{\circ}}$ together with  $C^{1,2}([0,T)\x \R)\cap  C^{0}([0,T]\x \R)$ functions $(\delta \bar V^{(i)})_{i=0,\cdots,i_{\circ}}$ such that, for $i=0,\cdots,i_{\circ}$, $\partial^{2}_{xx}\delta \bar V^{(i)}$ is $\delta \beta_{i}$-H\"older in space\bru{, uniformly on $[0,T)\x \R$,} and  $\delta \bar V^{(i)}$ solves
\begin{align*}
&\partial_{t}\delta \bar V^{(i)} +     \mu(\cdot,\check \ar_{\epsilon}) \partial_{x}\delta \bar V^{(i)}+\frac12 \sigma(\cdot,{\check \ar_{\epsilon}})^{2} \partial^{2}_{xx}\delta \bar V^{(i)} +r_{i}(\cdot,{\check \ar_{\epsilon}})  =0,\; \mbox{ on } [0,T)\times \R,  \\
&\delta \bar V^{(i)}(T,\cdot)=0\; \mbox{ on } \R,
\end{align*}
in which  $\check \ar_{\epsilon}$ is a Borel measurable map such that 
$$
\check \ar_{\epsilon}\in \underset{\bar a \in {\Ab}}{\rm argmax}  \left( \mu(\cdot,\bar a) \partial_{x}  \bar V^{(i_{\circ})}_{{\epsilon}}+\frac12 \sigma(\cdot,\bar a)^{2} \partial^{2}_{xx}  \bar V^{(i_{\circ})}_{{\epsilon}} +{r}(\cdot,\bar a) \right),
$$
with 
$$
\bar V^{(i_{\circ})}_{\epsilon}{:= }\delta \bar V^{(0)}+\sum_{j=1}^{i_{\circ} } \epsilon^{\frac{\beta_{{j}-1}}{{2}}} \delta \bar V^{({j})},\;\beta_{i}{:=}\sum_{j=0}^{i} \delta \beta_{j}\mbox{ for  } i\le i_{\circ},
$$
and, using the conventions $\delta \beta_{-1}:=0$ and $\delta r^{(-1)}_{\epsilon}:= r$, for $0\le i\le i_\circ$,
\begin{align}
 \delta r^{(i)}_{\epsilon}&:= \frac1\epsilon\bblue{ \int\left( \delta \bar V^{(i)}(\cdot,\cdot+b_{\epsilon})-\delta \bar V^{(i)} \right)\nu(\de e)}
+ \epsilon^{-\frac{\delta \beta_{i-1}}{2}}\delta r^{(i-1)}_{\epsilon} -  \bru{\mu}{\partial_{x}}\delta  \bar V^{(i)}-\frac12 \bru{\sigma}^{2} \partial^{2}_{xx}\delta \bar V^{(i)} -r_{i}\nonumber\\ 
r_{i} &:=r\1_{\{i=0\}}+\1_{\{i>0\}}\lim_{\epsilon \to 0}  \bblue{   \epsilon^{-\frac{\delta \beta_{i-1}}{2}} \delta r^{(i-1)}_{\epsilon}}\mbox{ for  } i\le i_{\circ}.\label{eq :def delta r i bis}
\end{align}
The limits in {\eqref{eq :def delta r i bis}} are well-defined on $[0,T\bru{)}\x \R$, and   
\begin{align}\label{eq: vitesse delta eps i circ}
 \limsup_{\epsilon\downarrow 0} \epsilon^{-\frac{\delta \beta_{i_{\circ}}}{2}}\bblue{\norm{  \epsilon^{-\frac{  \delta \beta_{i_{\circ}-1}}{2}}\delta r^{(i_{\circ}-1)}_{\epsilon}-r_{i_{\circ}}}_{\infty}}<\infty. 
\end{align} 
\end{Assumption}

\begin{Proposition} Let Assumption \ref{ass: existence systeme} hold. Then, for all $(t,x)\in [0,T]\x \R$, 
$$
\limsup_{\epsilon\downarrow 0} \epsilon^{-{\frac{\beta_{i_{\circ}}}{2}}} \abs{V_{{\epsilon}}-\bar V^{(i_{\circ})}_{\epsilon}}(t,x)<\infty.
$$
Moreover,  the control defined by
\begin{align*}
\check \alpha^{t,x}_{s}= \check \ar_{\epsilon}(s,X^{t,x,\check \alpha^{t,x}}_{s-}), \;s\in [t,T\bru{)},
\end{align*}
{satisfies} 
\begin{align*} 
\frac1{\lambda_{\epsilon}} \E\left[\int_{t}^{T}      r(X^{t,x,\check \alpha^{t,x}}_{s-},\check \alpha^{t,x}_{s}) \de N_{s}\right]
 \ge  V_{\epsilon}(t,x)-C\epsilon^{\frac{\beta_{i_{\circ}}}{2}}, \mbox{ { for all $\epsilon>0$,} }
\end{align*}
for some constant $C>0$.
\end{Proposition}
%

\begin{proof} With the above construction 
\begin{align*}
\partial_{t}\bar V^{(i_{\circ})}_{\epsilon}+  \frac1\epsilon \int \left(\bar V^{(i_{\circ})}_{\epsilon}(\cdot,\cdot+b_{\epsilon}(\cdot,\check \ar_{ {\epsilon}},e))-\bar V^{(i_{\circ})}_{\epsilon}(t,x)-\epsilon \epsilon^{\frac{\beta_{i_{\circ}-1}}{2}}\delta r^{(i_{\circ})}_{\epsilon}\bblue{(\cdot,\check \ar_{ {\epsilon}})}\right)\nu(\de e)+r(\cdot,\check \ar_{\epsilon})=0
\end{align*}
on $[0,T)\x \R$, while 
\begin{align*}
\partial_{t}\bar V^{(i_{\circ})}_{\epsilon}+  \frac1\epsilon \int \left(\bar V^{(i_{\circ})}_{\epsilon}(\cdot,\cdot+b_{\epsilon}(\cdot,  \ar,e))-\bar V^{(i_{\circ})}_{\epsilon}(t,x)-\epsilon \epsilon^{\frac{\beta_{i_{\circ}-1}}{2}}\delta r^{(i_{\circ})}_{\epsilon}\bblue{(\cdot,  \ar)}\right)\nu(\de e)+r(\cdot,  \ar)\le 0
\end{align*}
on $[0,T)\x \R$ for all $\ar:\bblue{\R\to \Ab}$.
 By \eqref{eq: vitesse delta eps i circ} and the same arguments as in the proof of Theorem \ref{thm : ecart V et bar V},
\begin{align*}
 \limsup_{\epsilon\downarrow 0} \epsilon^{-\frac{\delta \beta_{i_{\circ}}}{2}} \norm{\bblue{ \delta r^{(i_{\circ})}_{\epsilon}}}_{\infty}<\infty,
\end{align*}
so that the required result follows by verification. 
\end{proof}

\section{Application to an auction problem}\label{sec: exemple auction}

Repeated online auction bidding are typical problems in which the real value of the parameters $b,r$ and $\nu$ are unknown\bru{,} and on which reinforcement learning  techniques are applied. The later requires to estimate, very quickly, the optimal control for different sets of parameters. Being modeled as a discrete time problem, with fixed auction times, or more realistically  in the form of a pure-\bru{jump} problem as in Section \ref{sec : pure jump problem}, see also \cite{fernandez-tapia2017Optimal}, we face in any case the fact that auctions are issued almost continuously which corresponds to a very {small time step} in the discrete-time version or to a very large intensity in the pure-jump modelling. The numerical cost of a precise estimation of the optimal control is too important to combine it with a reinforcement learning approach.

\subsection{Model and description of the optimal policy}

We consider here a simple auction problem motivated by online advertising systems. A single ad-campaign is provided several opportunities to buy ad-space to display its ad over the course of the day. These ad spaces arrive at random, according to the point process $N$, since they are dependent on users from specific targeted audiences loading a website. In real-world display advertising, the kind encountered on the sides of web-pages, these opportunities take the form of an auction between several bidders and an ad-exchange platform. 

The format of the auction used is critical to the strategic behaviour of bidders and the revenue of the seller. There is a large amount of literature in auction theory on the subject, see \textit{e.g.} \cite{myerson1981optimal,ostrovsky2011reserve, paesleme2016field}, and real-world auctions can take very complex formats. For simplicity, we consider an auctioneer which has implemented a lazy second price auction \cite{vickrey1961counterspeculation, paesleme2016field} with individualised reserve price. In this format, our bidding agent wins the ad-slot if it submits a bid above its (henceforth the) reserve price and the competition, and if it wins it pays the maximum between the reserve price and the competition. For a given reserve price $x$, a bid $a\in(0,+\infty)$ and a random competition bid $B\ge 0$ following a smooth probability distribution $F_B$, the expected payoff $r(x,a)$ for an auction is thus expressible through a simple integration by parts as
\begin{align}
r(x,a) = \EE[(v-x\vee B)\1_{a\ge x\vee B}]=\1_{a\ge x}\left((v-a)F_B(a)+\int_x^aF_B(b)\de b \right)\,,\label{eq:exemple reward} 
\end{align}
in which $v$ is the value of the ad-slot for the bidder.  \bblue{Note that $r$ is not continuous  as it is assumed in the preceding sections. In practice, one can replace it by a smooth approximation. In the following, we shall construct a numerical scheme directly on $r$, without smoothing. It turns out that convergence still seems to be observed at the rate $\epsilon^{\frac12}$. Intuitively,  this is due to  the fact that the maximum values obtained in \eqref{eq: PDE HJB Jump} and \eqref{eq: PDE bar V} are the same 
for $r$ defined with $\1_{a\ge x}$ and $\1_{a>x}$ whenever $x<\sup \Ab$, which is true  at each time with probability one for the controlled processes defined below.}

As the right hand side of \eqref{eq:exemple reward} highlights, reserve prices are a mechanism put in place by sellers to compensate for lack of competition, which would drive down the price and their profits. It is well established that a reserve price is not as profitable as increasing the number of participants by one \cite{Bulow1996Auctions}. Consequently, when there are many bidders a control will have little effect on the system. To clearly demonstrate the use of controlling the reserve price, we study a strongly asymmetric setting, where the agent has a value $v=0.5$ much higher than the competition, which we take uniform on $(0,0.3)$. In this setting, it is directly competing against the seller for \lc{its} extra value above the average competition. For the purpose of this example, we do not want to go to the limit of this asymmetry, the posted price auction where there is no competition, as it could lead the control problem to degeneracy, such as negative prices and difficult boundary conditions.

There is a large literature on revenue maximisation algorithms in online auctions, or how to set the reserve price to maximise revenue, such as \cite{croissant2020real,cesa2014regret,blum2004online,bubeck2019multi}.
For the sake of simplicity, in this example, we will model the dynamics of the reserve price using a simple mean reverting process:
\[b_{1}(x,a,e) = \kappa a+ (1-\kappa)r_0 -x \mbox{ and } b_{2}(x,a,e)=e \mbox{ with } \nu\sim\texttt{Unif}(-0.1,0.1), \] 
with $\kappa\in (0,1)$ and $r_{0}\in\RR_+$. The reserve price process $X^{t,x,\alpha}$ is then defined from these coefficients as in \eqref{def : X}, with $b:=b_\epsilon=\epsilon b_1 +\sqrt\epsilon b_2$ and $\lambda:=\lambda_\epsilon=\epsilon^{-1}$. This corresponds to setting a minimum reserve price $(1-\kappa)r_0$, and tracking the agent's bid with aggressiveness measured by $\kappa$. Setting $r_0=0.15$ as the monopoly price of the competition guarantees the seller a better revenue against the competition, while $\kappa a$ allows him to pursue the agent's extra value. We set $\kappa=1/2$, for a balance between prudence and aggression.

The control problem consists in maximising the static auction revenue, while considering the impact bids have on the system. In the static auction format, we can identify three domains the reserve price can be in: ``non-competitive'', ``competitive'', and ``unprofitable''. When the reserve price is below the competition's average\footnote{Recall that the competition here models the distribution of the maximum bid of all other participants, so this is the average of the maximum of other participants' bids.} there is essentially no prejudice to the agent, since the reserve price barely affects his profits. Therefore there is no need to compete with and control the reserve price. On the other hand, when the reserve price is in the range between $0.3$ and $v=0.5$, the reserve price becomes the dominant term in $r$ and the agent has to compete with the seller over the value margin it has relative to other buyers. Finally, if the reserve price is above $v$, there is no possible profit so no reason to take part in the auction by bidding $a>0$. \bru{For the same reason, we take $ {\Ab:=}[0,\bru{0.5]}$.}

When the reserve price is dynamic, a good control seeks to maximise profit while pushing the reserve price to the non-competitive domain. \lc{One} can see this in effect on figure \ref{fig:trajectories}. In the non-competitive regime (left), starting at a reserve price of $0.15$, this policy recovers $85\%$ of the best possible income of the static setting, where the reserve price is $0$ for all $t$, and the average price is $0.5-\EE[B]=0.35$. In the competitive regime (centre), the policy bids just above the reserve price to apply a downwards pressure until it reaches the non-competitive domain again. Finally, in the unprofitable regime (right), the agent boycotts the auction by bidding $0$, bringing down the price. Notice how when the agents stops boycotting there is an inflection point in the downwards trend of the price, schematically represented by the dotted line. 
\begin{figure}[ht]
    \centering
    \includegraphics[width=\textwidth]{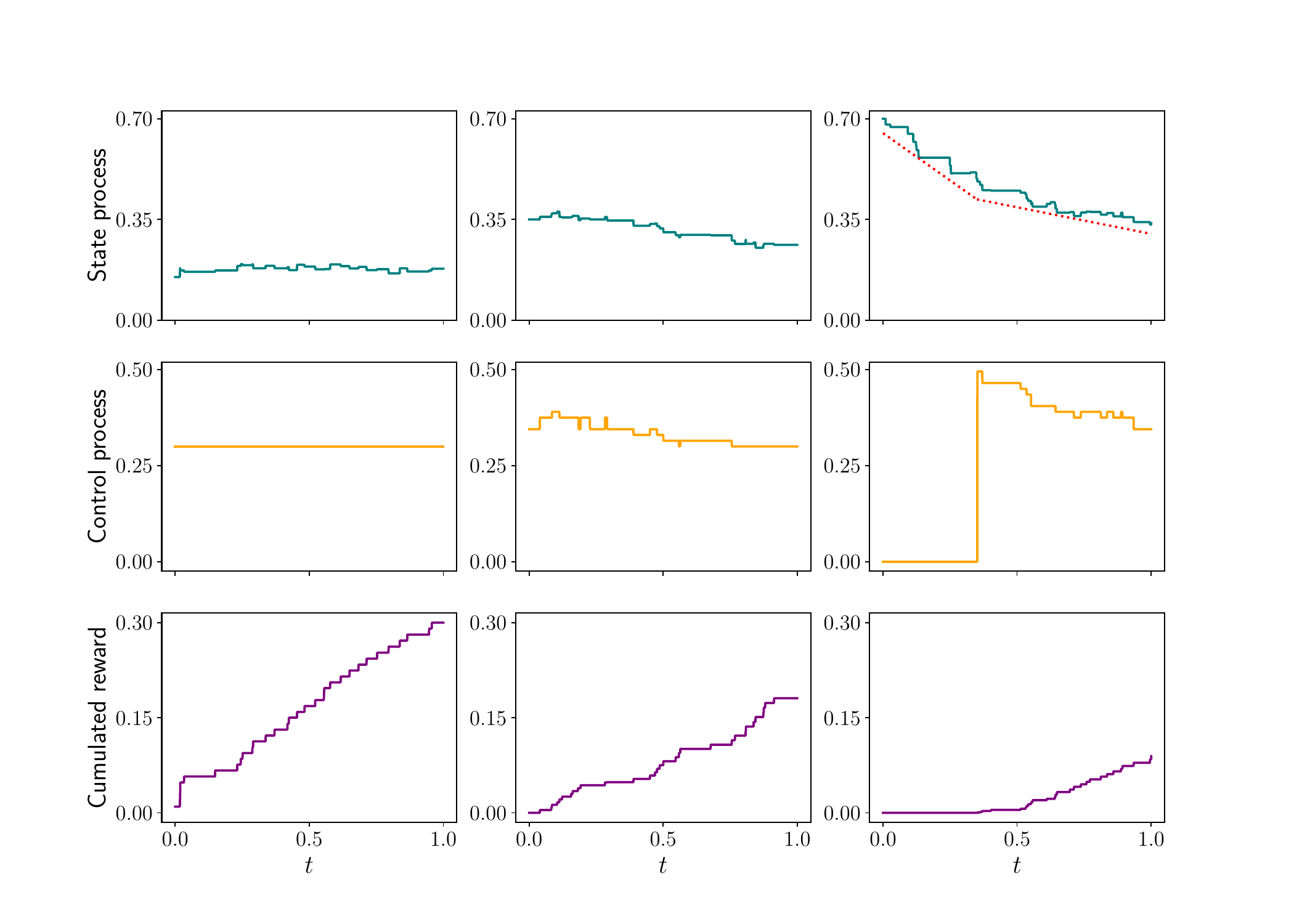}
    \caption{Selected sample realisations of the system for $\epsilon=10^{-1.5}$, starting from ${x=}0.15$ (left), ${x=}\bru{0.35}$ (centre), and ${x=}0.7$ (right).} \label{fig:trajectories} 
\end{figure}

\subsection{Numerical implementation}\label{subsec: example: pure jump}

Adapting \eqref{eq: PDE HJB Jump} and \eqref{eq: PDE bar V}, we normalise the horizon to $1$, and allow the reserve price to \bru{vary in $\RR$.} This allows us to easily set boundary conditions for the equation. When an auction happens with a negative price $x$, the price is set by the competition, which will be a.s. positive. \lc{Thus} as $x\to-\infty$, the reserve price becomes irrelevant and the value converges to the value of a single auction \bru{without reserve price.} Conversely, for $\epsilon<1$, as $x\to+\infty$, the probability of $X_t^{0,x,\alpha}$ descending below $v$ by time $T$ and generating any revenue decreases due to the noise. \bblue{Hence, a} \bru{Neuman} boundary condition \bblue{set to 0} is appropriate at $[0,1)\times\{-\infty,+\infty\}$. In numerical resolution, we will use \bru{Neuman boundary conditions equal to 0 on $[0,1)\times\bblue{\{-1,3\}}$}. Given this domain for the reserve price, we can set the controls on an even mesh in $\bblue{\Ab=}[0,\bru{0.5]}$, of fineness $0.01$. 

We solve both problems numerically with an explicit finite difference solver, and for simplicity a Riemann sum using the same mesh for the numerical integration part. This formulation is equivalent to a Markov Chain control problem.  Let $M_t=\{k\Delta_t\,;\, k=0,\dots,\lfloor1/\Delta_t\rfloor\}$, $M_x=\bblue{\{-1 +k\Delta_x\,;\, k=0,\dots,\lfloor4/\Delta_x\rfloor\}}$ be the  \bblue{time and space} meshes, with \lc{finenesses}, $\Delta_x=\epsilon^{3/2}/2$, $\Delta_t=\Delta_x^{2/3}$. Denote $V_n(x_i)$ the output of the solver at time $t_n\in \mc{M}_t$ and position $x_i\in\mc{M}_x$. For the pure jump problem, we explicitly compute:
\[
V_n^\epsilon(x_i)=V_{n+1}^\epsilon(x_i) + \frac{\Delta_t}{\epsilon}\sup_{a\in\mathbb{A}_n} \left(\sum_{x_j\in \mc{M}_x} V_{n+1}^\epsilon(x_j)f^{\nu,\epsilon}_{x_i,a}(x_j)\Delta_x - V_{n+1}^\epsilon(x_i) + r(x_i,a)\right)\,
\]
where $f^{\nu,\epsilon}_{x,a}$ is the transition kernel induced by $b_1(x,a,\cdot)$, $b_2(x,a,\cdot)$, and $\nu$. For the diffusion, we consider meshes $\mc{M}_t=\{kd_t\,;\, k=0,\dots,\lfloor1/d_t\rfloor\}$, $\mc{M}_x=\bblue{\{-1 +kd_x\,;\, k=0,\dots,\lfloor4/d_x\rfloor\}}$, with $d_x=10^{-2}$, $d_t=d_x^2$ and solve recursively
\[ \bar V_{n-1}(x_i) = \bar V_n(x_i) + d_t \sup_{a \in \mathbb{A}_n}\left((\kappa a +(1-\kappa) r_0 -x)\delta^u_x\bar V_n(x_i) + \frac{1}{2}\sigma^2\delta_{xx}\bar V_n({x_{i}}) + r(x_i,a)\right)\]
where $\delta^u_x$ and $\delta_{xx}$ are the uplift first order and centred second order finite differences on $\mc{M}_x$ respectively. We took $\mathbb{A}_n=\{10^{-2}k; k=0,\dots,\bblue{50}\}$   in both cases.

To give some insight into the complexity trade-off, see that\bru{,} when $\epsilon$ is large, there are relatively few jumps so the time iteration won't require many steps to get an accurate solution. This scaling is indicated by the $\Delta_t/\epsilon$ term. At the same time, the jumps are large so even a coarse mesh in $x$ will be sufficient for the numerical integration to approach the integral. Unfortunately as $\epsilon\to0$, one must refine both the time mesh, linearly with $1/\epsilon$, and the integration mesh which is paid quadratically due to the non-local nature of the equation. In practice, this makes computations grow at a super-cubic rate with $\epsilon$, which becomes prohibitively expensive quickly. In our example problem, the noise is supported on a bounded interval of size $\sqrt{\epsilon}$, \lc{and one thus saves} some computation time, but \lc{F}igure \ref{fig:comp cost} shows the computation cost (pictured with dots) still grows super-quadratically and overcomes the cost of our accurate diffusion mesh  (solid horizontal line) even for large $\epsilon$. Even though we computed the diffusive limit to a very high precision, and with an explicit scheme, for $\epsilon$ of order of $10^{-3}$ the CPU time spent \lc{on resolution} is already $6$ times higher in the pure-jump problem. Note that, in the pure-jump case, if the control were to intervene in a non-linear way we might need to also refine the control mesh with $\epsilon$, further increasing the computational burden. 

\begin{figure}
\centering
\includegraphics[width=.75\textwidth]{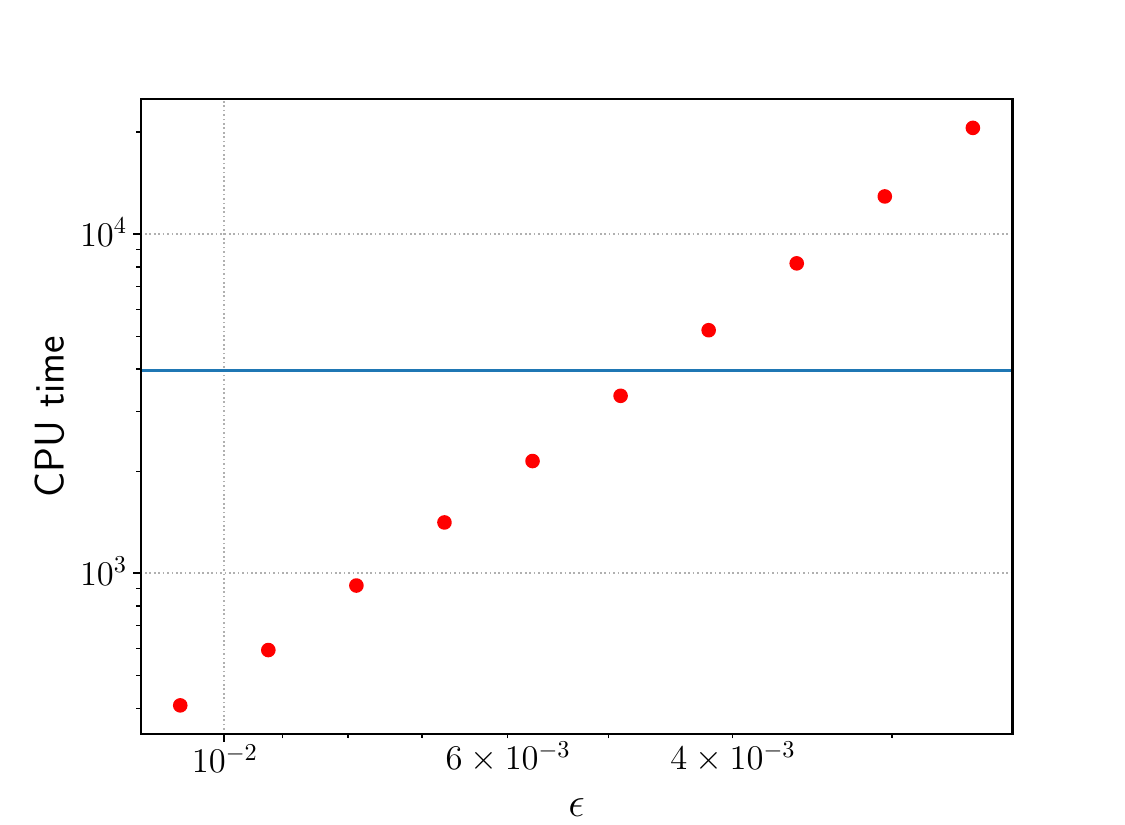}
\caption{Numerical cost for $V_{\epsilon}$ (log scales).}
\label{fig:comp cost}
\end{figure}


Beyond gains in computation, Figure \ref{fig:value functions} verifies that Proposition \ref{thm : ecart V et bar V} holds with meaningful constants in finite time on this problem. Figure \ref{fig:value functions} shows that the error is very low even for large values of $\epsilon$, and decreases at the correct rate of $\epsilon^{\bblue{1/2}}$. Likewise, Figure \ref{fig:policies} shows the rate of Proposition \ref{prop: construction control eps opti} also holds even for large $\epsilon$. 

\begin{figure}[ht]
\centering
\begin{minipage}[t]{.45\textwidth}
    \centering
    \includegraphics[width=\textwidth]{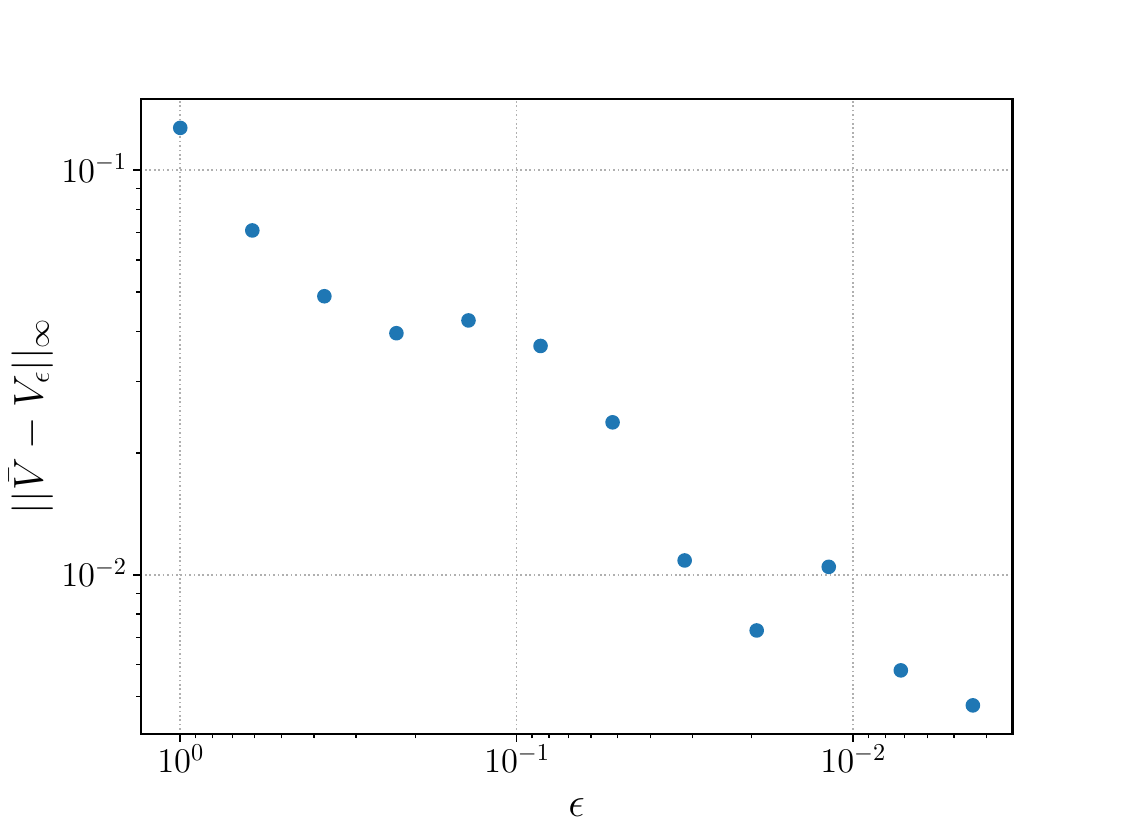}
    \caption{Limit value function error relative to $V_\epsilon$, at $t=0$ (log scales).}\label{fig:value functions}
\end{minipage}
~~
\begin{minipage}[t]{.45\textwidth}
    \centering
    \includegraphics[width=\textwidth]{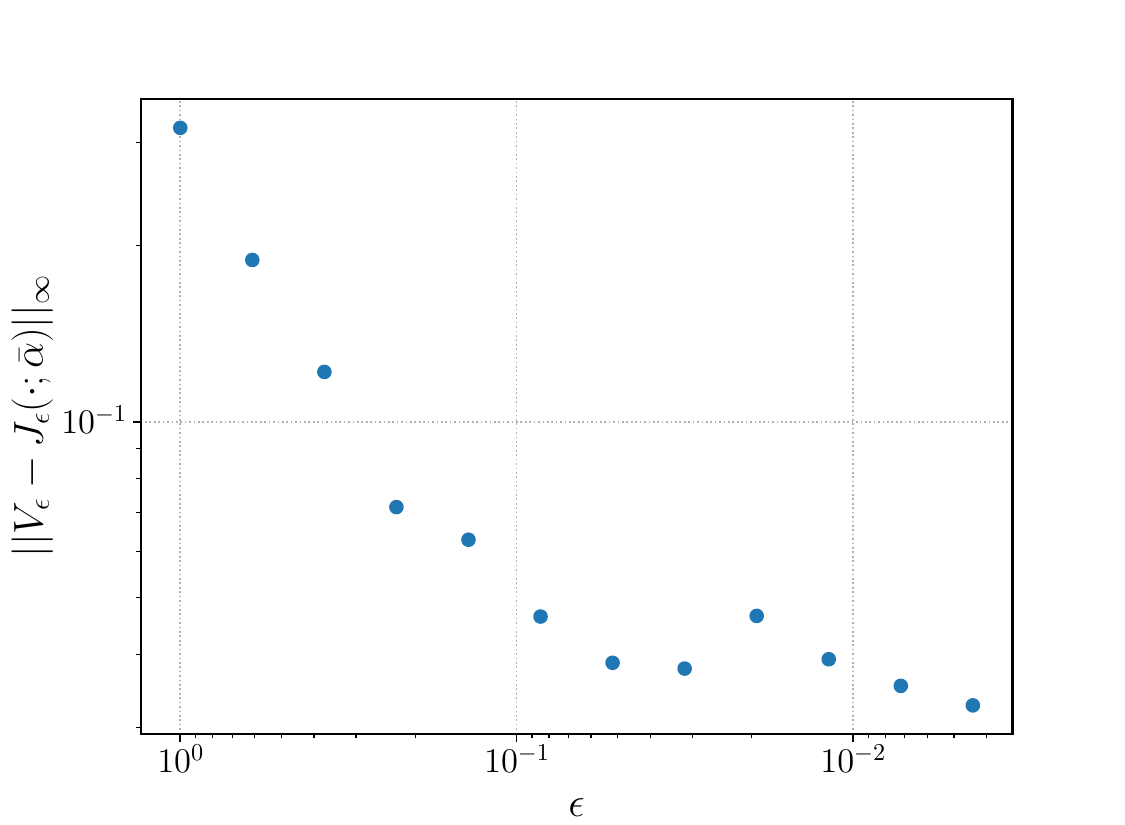}
    \caption{Limit policy error relative to $V_\epsilon$, at $t=0$ (log scales).}\label{fig:policies}
\end{minipage}
\end{figure}


 \section{Remark on the diffusive limit of discrete time problems}\label{subsec: discrete time problem}

Instead of considering the diffusive limit of a continuous time pure-jump problem, one could similarly consider a sequence of pure discrete time problems with actions at time $t^{n}_{i}:=iT/n$, $i\le n$:
\begin{align*}
V_{n}(t,x)&:=\sup_{\alpha \in \Ac}\frac{T}{n}\E\left[ \sum_{i\bru{=1}}^{n}\1_{\{t^{n}_{i}\ge t\}}   r(X^{t,x,\alpha}_{t^{n}_{i}-},\alpha_{t^{n}_{i}})\right],
\end{align*}
with $X^{t,x,\alpha}$ defined by
\begin{align*}
X^{t,x,\alpha}=\bru{x+ }  \sum_{i\bru{=1}}^{n}\1_{\{t^{n}_{i}\in (t,\cdot]\}}   b(X^{t,x,\alpha}_{t^{n}_{i}-},\alpha_{t^{n}_{i}},\xi_{i}^{n} )
\end{align*}
and in which $(\xi^{n}_{i})_{i\ge 1}$ is i.i.d.~following the distribution $\nu$ and $\Ac$ is the collection of $\Ab$-valued processes that are predictable with respect to the $\P$-augmented filtration generated by 
$
  \sum_{i\bru{=1}}^{n}\1_{\{t^{n}_{i}\in [0,\cdot]\}}    \xi_{i}^{n} .
$

Upon taking $b$ of the form 
$$
b_{n}=\frac{T}{n}b_{1}+\sqrt{\frac{T}{n}}b_{2},\;\mbox{ with } \E[b_{2}(\cdot,\xi^{n}_{1})]=0, 
$$
one would obtain the same diffusive limit as in Section \ref{subsec : conv speed} when letting $n\to \infty$. Namely, the same arguments as in \cite[Section 3]{fleming1989existence} combined with Proposition \ref{prop: HJB bar V} and the fact that comparison holds for \eqref{eq: PDE bar V} imply that 
$\lim_{n\to \infty} V_{n}$ is well-defined and is equal to $\bar V$. 

One can also check that the convergence holds at a speed $n^{-\frac{\beta}{2}}$. Let us sketch the proof.  First, the same arguments as in the proof of Theorem \ref{thm : ecart V et bar V}  imply that 
\begin{align*}
\delta r_{n}:=
\frac{n}{T}\bblue{\E\left[\bar V(\cdot,\cdot+b_{n}(\cdot,\xi_{i}^{n} ))-\bar V\right]}- \bblue{\mu} \partial_{x} \bar V-\frac12 \bblue{\sigma^{2}}\partial_{xx} \bar V
\end{align*}
satisfies
\begin{align}\label{eq: borne delta r T/n} 
\norm{\delta r_{n}}_{\infty}\le Cn^{-\frac{\beta}{2}}
\end{align}
for some $C>0$ independent on $n$. Thus, by Proposition \ref{prop: HJB bar V}
\begin{align*}
0&=\partial_{t} \bar V(t,x)\frac{T}{n}+  \sup_{a\in \Ab}\bblue{ \E \left[ \bar V(t,x+b_{n}(x,  a,\xi^{n}_{1}))-\bar V(t,x)+\frac{T}{n}(r(x,  a)-{\delta r_{n}(x, a)})\right]}
 \end{align*}
so that 
\begin{align*}
\bar V(t^{n}_{i},x)=& \sup_{a\in \Ab}\E\left[\int_{t^{n}_{i}}^{t^{n}_{i+1}} \partial_{t} \bar V(t^{n}_{i},x) \de s + \bar V(t^{n}_{i},x+b_{n}(x,  a,\xi^{n}_{i+1}))+\frac{T}{n}(r(x,  a)-\bblue{\delta r_{n}(x,  a)}) \right]\\
&=  \sup_{a\in \Ab} \left(\E\left[ \bar V(t^{n}_{i+1},x+b_{n}(x,  a,\xi^{n}_{i+1}))+\frac{T}{n} r(x,  a) \right]\right.\\
&~~~~~~~~~~+ \left.\E\left[\int_{t^{n}_{i}}^{t^{n}_{i+1}} (\partial_{t} \bar V(t^{n}_{i},x)-\partial_{t} \bar V(s,x+b_{n}(x,  a,\xi^{n}_{i+1})) -\bblue{\delta r_{n}(x,  a)} \de s \right]\right).
 \end{align*}
We then use \eqref{eq: holder vt}  and \eqref{eq: borne delta r T/n} to obtain that 
\begin{align*}
\bar V(t^{n}_{i},x)& = \sup_{a\in \Ab} \E\left[ \bar V(t^{n}_{i+1},x+b_{n}(x,  a,\xi^{n}_{i+1}))+\frac{T}{n} r(x,  a) +\int_{t^{n}_{i}}^{t^{n}_{i+1}} \varpi_{n}(s,x,a) \de s \right]
 \end{align*}
in which $|\varpi_{n}|\le Cn^{-\frac{\beta}{2}}$, for some $C>0$ independent of $n$. It follows that 
$$
\bar V(t^{n}_{i},x)=\sup_{\alpha \in \Ac}\E\left[\bru{\frac{T}{n} \sum_{j=i}^{n}  r(X^{t,x,\alpha}_{t^{n}_{j}-},\alpha_{t^{n}_{i}}) +\int_{t^{n}_{i}}^{T} \varpi_{n}(s,{X^{t,x,\alpha}_{s}},{\alpha_{s}}) \de s } \right],
$$
which provides the expected result \bru{since $\partial_{t}\bar V$ is bounded}.

Likewise, the Markovian control defined through \eqref{eq: def bar a theta} can be shown to be $n^{-\frac{\beta}{2}}$-optimal for $V_{n}$, see the proof of Proposition \ref{prop: construction control eps opti}.

\section{\bru{Conclusion}}

We studied the diffusion limit of a pure-jump control problem as the jump intensity goes to infinity, upon assuming a correct scaling of the coefficients.  Under appropriate conditions, we showed that the second order derivative of the value function associated to the limiting diffusing problem is H\"older continuous and that its H\"older exponent drives the convergence rate. Convergence can even be improved by using a first (or even higher) order correction scheme. This approach is particularly efficient for the numerical approximation of the optimal control associated to a pure jump process with large intensity, as it is the case in auctions associated to online advertising systems.

\bibliographystyle{plain}
\def\cprime{$'$} \def\cprime{$'$}

\end{document}